\documentclass[sn-mathphys,Numbered]{sn-jnl} 

\usepackage{anyfontsize}
\usepackage{graphicx} 
\usepackage{multirow} 
\usepackage{amsmath,amssymb,amsfonts} 
\usepackage{amsthm} 
\usepackage{mathrsfs} 
\usepackage[title]{appendix} 
\usepackage{xcolor} 
\usepackage{float}
\usepackage{textcomp} 
\usepackage{manyfoot} 
\usepackage{booktabs} 
\usepackage{algorithm} 
\usepackage{algorithmicx} 
\usepackage{algpseudocode} 
\usepackage{listings}
\usepackage{array}
\usepackage{pdflscape}
\usepackage{caption}
\usepackage{subcaption}
\usepackage[utf8]{inputenc}
\usepackage{booktabs}
\usepackage{amsfonts}
\usepackage{siunitx}
\usepackage{longtable}
\usepackage{eucal}
\usepackage{multirow}
\usepackage{enumerate}
\usepackage{lineno}
\usepackage{framed}
\usepackage{placeins}
\usepackage[disable]{todonotes}

\newtheorem{theorem}{Theorem}[section]
\newtheorem{proposition}{Proposition}[section] 
\newtheorem{lemma}{Lemma}[section] 
  
\newtheorem{example}{Example}[section] 
\newtheorem{remark}{Remark}[section] 
\newtheorem{assumption}{Assumption}[section] 
\newtheorem{definition}{Definition}[section]
\begin{document}
\title[Hager-Zhang Conjugate Gradient Method of Set Optimization]{Hager-Zhang Conjugate Gradient Method for Set Optimization with Set-Valued Objective Map of Finite Cardinality}

\author[1]{\fnm{Ravi} \sur{Raushan}}\email{raviraushan.rs.mat21@itbhu.ac.in}

\author*[1]{\fnm{Debdas} \sur{Ghosh}}\email{debdas.mat@iitbhu.ac.in}

\author[2,3]{\fnm{Zai-Yun} \sur{Peng}}\email{pengzaiyun@126.com}

\affil[1]{\orgdiv{Department of Mathematical Sciences}, \orgname{Indian Institute of Technology (BHU)}, \orgaddress{\city{Varanasi}, \postcode{221005}, \state{Uttar Pradesh}, \country{India}}}

\affil[2]{\orgdiv{School of Mathematics}, \orgname{Yunnan Normal University}, \orgaddress{\city{Kunming}, \postcode{650500},  \country{P.R. China}}}

\affil[3]{\orgdiv{College of Mathematics and Statistics}, \orgname{Chongqing JiaoTong University}, \orgaddress{\city{Chongqing}, \postcode{400074},  \country{P.R. China}}}

\abstract{This work introduces a nonlinear Hager-Zhang conjugate gradient method for solving set optimization problems. The objective function under consideration is defined by a finite collection of continuously differentiable functions. 
Notably, the proposed approach imposes restrictions neither on the existence of a finite generator of the ordering cone nor  
on any regularity condition at the optimal solution.  As a result, the proposed method holds considerable significance for both set optimization and vector optimization problems, with the latter serving as a special case of the former. 
The study begins by discussing Wolfe line search conditions using Drummond-Svaiter scalarization function. Thereafter, we establish the existence of a step length satisfying the Wolfe line search conditions along a descent direction. The Hager-Zhang scalar conjugate parameter is introduced to derive the search direction for the proposed method. It is established that the direction generated by the proposed method is a descent direction. The well-definedness of the proposed method is given. Furthermore, we discuss some important results and a Zoutendijk-like condition to ensure global convergence. Subsequently, the global convergence of the proposed method is established in an asymptotic manner. Finally, numerical experiments on various test problems validate the practical performance and effectiveness of the proposed technique. \\


\smallskip

}

\maketitle
\section{Introduction}
Set optimization problems generalize traditional optimization by considering set-valued maps as objectives or constraints rather than just scalar or vector-valued functions. These problems naturally arise in scenarios where the objective or constraint functions are adhered to uncertainty, in the duality theory of multi-criteria decision-making, or by the intrinsic structure of the problem, e.g., portfolio optimization in finance, where the objective map represents the set of achievable returns for different risk levels. In such optimization problems, the general goal is to optimize a set-valued map \( F: X \to \mathscr{P}(Y) \), where \( X \) is the decision space, \( Y \) is the outcome space, and \(\mathscr{P}(Y) \) represents the power set of \( Y \).\\

Set optimization is significant for its adaptability and diverse range of applications. It plays a critical role in:
\begin{itemize}
    \item Game theory: Modeling strategies with payoffs as sets of possible outcomes, aiding in robust decision-making.
    \item Economics: Addressing multi-agent systems where objectives or constraints have overlapping or uncertain outcomes.
    \item Mathematical finance: Handling portfolio optimization under uncertainty by considering sets of potential returns.
    \item Robust optimization: Capturing worst-case or uncertain scenarios by modelling parameters as sets. 
\end{itemize}
Detailed references on the mathematical and practical applications of set optimization problems are available in \cite{Ghosh2024newton,Kishor Weighted Aggregation 2024,khan2016set}.\\

Solution concepts in set optimization are commonly defined using two approaches: the vector approach and the set approach. The primary limitation of the vector approach is its inability to address scenarios where the decision-maker's preference relies on comparing individual elements within the image set. To address this, the set approach provides a more versatile alternative.
Kuroiwa, in \cite{Kuroiwa1996min}, introduced the set approach, which employs a pre-order relation on the power set of the image space to define minimal solutions. 
This method has been widely adopted by researchers, as demonstrated in works such as \cite{khan2016set, Vector2012optimization, Variational2019analysis,
Characterizations2018relations, 
steepest2021set_optimization} and their references.
In this article, we also adopt the set approach with the lower set pre-order $\preceq^{\ell}$ corresponding to a given ordering cone $K\subset Y$ as proposed in \cite{Kuroiwa1996min} and find the weakly minimal set accordingly. The ordering cone $K$ is assumed to be closed, convex, solid (interior of $K$ is non-empty), and pointed ($K\cap (-K)=\{0\}$).  
The lower set pre-order $\preceq^{\ell}$ is defined as follows: for $U,V\in \mathscr{P}(Y)$,
\[
U\preceq^{\ell}V \;\iff\; V\subseteq U+K.
\]
Similarly, the strict lower pre-order $\prec^{\ell}$ is defined by
\[
U\prec^{\ell}V \;\iff\; V\subseteq U+\mathrm{int}(K),
\]
where $\mathrm{int}(K)$ denotes the interior of the cone $K$. \\

There are several approaches to solving set optimization problems and identifying solutions. These methods are generally classified into five main groups: sorting-type methods, derivative-free methods, scalarization-based methods, branch and bound approach, and first-order solution methods. The first four types of methods have some drawbacks, as highlighted in \cite{steepest2021set_optimization}. This work endeavours to propose a method grounded in first-order solution techniques.\\

In this work, we are interested in the following \emph{unconstrained set-valued optimization problem with respect to the lower set pre-ordering} 
\begin{equation}\label{SPL}
    \text{($\mathcal{SOP}_{\ell}$)} ~~\preceq^{\ell}\text{-}~\underset{~x\in \mathbb{R}^{n}}{~\min}F(x),
\end{equation}
where the set-valued map $F:\mathbb{R}^{n}\rightrightarrows\mathbb{R}^{m}$ is defined as 
$$ F(x):=\left\{f^{1}(x),f^{2}(x), \ldots, f^{p}(x)\right\},$$ where $f^{1},f^{2},\ldots,f^{p}:\mathbb{R}^{n}\rightarrow \mathbb{R}^{m}$ are continuously differentiable function.\\

Note that the $\mathcal{SOP}_{\ell}$ \eqref{SPL} is an extension of an unconstrained vector optimization problem. The concept of a weak Pareto-optimality has to be replaced with a weakly minimal solution. A point $\bar{x}\in \mathbb{R}^{n}$ is said to be a weakly minimal solution (local weakly minimal solution) if there is no $x$ in $\mathbb{R}^{n}$ (no $x$ in a neighbourhood of $\bar{x}$) such that $F(x)\prec^{\ell} F(\bar{x})$.\\

The proposed technique in this work is based on non-linear conjugate gradient methods, which have emerged as an important class of first-order algorithms for solving large-scale optimization problems, primarily due to their fast convergence and low memory requirements. Originally introduced in \cite{FR_1964} for conventional optimization problems, these methods have evolved into several notable variants, including the Hestenes-Stiefel (HS, 1952) method \cite{HS_1952}, the Fletcher-Reeves (FR, 1964) method \cite{FR_1964}, the Polak-Ribi\`ere-Polyak (PRP, 1969) method \cite{PRP_1969}, the Conjugate Descent (CD, 1987) method \cite{CD_1987}, the Dai-Yuan (DY, 1999) method \cite{DY_1999}, and the Hager-Zhang (HZ, 2005) method \cite{HZ_2005}. 
In \cite{Nonlinear2018conjugate}, P\'erez and Prudente extended conjugate gradient methods to vector optimization by introducing non-linear conjugate gradient methods. They adopted the Wolfe line searches and a Zoutendjik-like condition for vector-valued optimization and explored the FR, CD, DY, PRP, and HS methods within the framework of finitely generated cones. This work laid the foundation for further developments in the field.
In recent years, additional advancements have been made. In 2020, Gonçalves and Prudente extended the Hager-Zhang method \cite{extension2020Hager-Zhang} for vector optimization, demonstrating its global convergence and superior performance compared to earlier variants on many test problems. Building on this progress, Hu et al. introduced an alternative extension of the Hager-Zhang method \cite{hu2024alternative} in 2024, which showed further performance improvements. Like earlier approaches, these methods also operate within the framework of finitely generated cones. \\

Recently, Bouza et al. proposed the steepest descent method \cite{steepest2021set_optimization} for \(\mathcal{SOP}_{\ell}\) \eqref{SPL}, assuming regularity at the optimal point. They have developed optimality conditions for these types of problems and introduced concepts of critical points. These optimality conditions are based on the equivalence relation between \(\mathcal{SOP}_{\ell}\) \eqref{SPL} and a family of vector optimization problems at a given weakly minimal point. 
Under the same assumption, Ghosh et al. introduced Newton method \cite{ghosh2024newton} for \(\mathcal{SOP}_{\ell}\) \eqref{SPL}. For the same problem, Kumar et al. \cite{kumar2024nonlinear} developed a non-linear conjugate gradient method with two variants, FR and CD, by using Drummond-Svaiter scalarization, also based on the regularity assumption at the optimal point. In \cite{kumar2024nonlinear}, for a finitely generated cone \(K\), Wolfe's line search and a Zoutendijk-like condition were also established for \(\mathcal{SOP}_{\ell}\) \eqref{SPL}. 
More recently, Ghosh et al. \cite{Ghosh CGM} proposed a non-linear conjugate gradient method with three additional variants, DY, HS, and PRP, for \(\mathcal{SOP}_{\ell}\) \eqref{SPL}, notably without requiring the regularity assumption at the optimal point. The approach in \cite{Ghosh CGM} utilized an alternative generator and the Gerstewitz scalarizing function, introduced Wolfe's line search and Zoutendijk-like conditions, and established the convergence of the proposed method without the assumption of a finitely generated cone \(K\).\\

Notably, the Hager-Zhang non-linear conjugate gradient method \cite{hu2024alternative} has demonstrated exceptional performance compared to other variants in vector optimization. 
Motivated by these developments, this work aims to extend the Hager-Zhang conjugate gradient method for \(\mathcal{SOP}_{\ell}\) \eqref{SPL} using Drummond-Svaiter scalarizing function. This approach, like the one in \cite{Ghosh CGM}, does not rely on the assumption of finitely generated cones or impose regularity conditions on the optimal point. 
It is worth emphasizing that unconstrained vector optimization represents a particular case of \(\mathcal{SOP}_{\ell}\) \eqref{SPL}. Consequently, all the results derived in this study are applicable to vector optimization problems, potentially contributing to the improvement of methods in this domain.\\

This paper is structured as follows. The subsequent section presents the notations, basic definitions, and foundational results that are used throughout the study. In Section \ref{sec_3}, Wolfe line searches conditions and the existence of a step-size that satisfies these line search conditions are discussed.
 Additionally, the non-linear Hager-Zhang method to solve $\mathcal{SOP}_{\ell}$ \eqref{SPL} and its well-definedness are presented in Section \ref{sec_3}. 
Section \ref{sec_4} demonstrates that the proposed method is globally convergent. Numerical results illustrating the performance of the method are provided in Section \ref{sec_5}. Finally, Section \ref{sec_6} concludes the paper with a summary and suggestions for future research directions.

\section{Preliminaries and terminologies}

This section establishes the supporting notations and fundamental results that form the theoretical basis for the set optimization problems studied in the paper.
\\

For any $p \in \mathbb{N}$, we denote the set $\{1, 2, \dots, p\}$ by $[p]$. The notation $\mathbb{R}^m_{+},~ \mathbb{R}^m_{++},~\mathscr{P}(\mathbb{R}^m)$ refers to the nonnegative orthant, positive orthant, and power set of $\mathbb{R}^m$, respectively.
For a set $A \in \mathscr{P}(\mathbb{R}^m)$, $\mathrm{conv}(A)$ denotes the convex hull of $A$, and $\mathrm{int}(A)$ denotes the topological interior of $A$.
\\

We now discuss some definitions and properties of the norm of the vector and the matrix.\\

Given $x\in \mathbb{R}^{n}$. By the notations $\|x\|_{1}$, $\|x\|_p$, and $\|x\|_\infty$, we denote the quantities $\sum_{i=1}^{n}|x_{i}|$, $\left(\sum_{i=1}^{n}{\lvert{x_{i}\rvert}^{p}}\right)^{{1}/{p}}$, and $\max \left\{|x_1|, |x_2|, \ldots, |x_n|\right\}$, respectively, where $p\in \mathbb{N}$. These notations about the norm of a vector also imply the following generalized version of the Cauchy-Schwarz inequality:
\[\forall x,y\in \mathbb{R}^{n}:~ x^{\top}y \leq \|x\|_{p}\|y\|_{q},\]
 where $p,q\geq 1$ and $\frac{1}{p}+\frac{1}{q}=1.$ For simplicity, we denote the $\ell_{2}$-norm on $\mathbb{R}^{n}$, i.e., $\|\cdot\|_{2}$ by $\|\cdot\|$ in this paper.\\

 Let $\|\cdot\|_{a}$ and $\|\cdot\|_{b}$ be two norms on $\mathbb{R}^{n}$ and $\mathbb{R}^{m}$, respectively. For a given matrix $A\in \mathbb{R}^{m\times
n}$, the induced matrix norm $\|A\|_{ a,b}$ is defined by
\[\|A\|_{ a,b}: =\underset{x\in \mathbb{R}^{n}}{\max}\left\{\|Ax\|_{b}~\middle|~\|x\|_{a}\leq 1\right\}.\] 
The above definition implies the following inequality:
\[\|Ax\|_{b}\leq \|A\|_{a,b}\|x\|_{a} \text{ for all }x\in \mathbb{R}^{n}.\]
For simplicity, if $a=b$, we use $\|\cdot\|_{a}$ instead of $\|\cdot\|_{a,a}$. The above definition for the induced norm of $A$, we have the following results for $a=b$.
\begin{itemize}

    \item $\|A\|_{1}:=\underset{j\in[n]}{\max}\sum_{i=1}^{m} \lvert A_{i,j}\rvert$, where $A
    _{i,j}$ represents the element at the $(i,j)$
 position of the matrix $A$.\\

     \item $\|A\|_{2}:=\left(\lambda_{\max}(A^{\top}A)\right)^{\frac{1}{2}}$, where $\lambda_{\max}(A^{\top}A)$ is the maximum eigenvalue of the matrix $A^{\top}A$.\\

 \item $\|A\|_{\infty}:=\underset{i\in[m]}{\max}\sum_{j=1}^{n} \lvert A_{i,j}\rvert$.\\
\end{itemize}
\noindent
Moreover, by the definitions of $\|A \|_2$ and $\|A\|_{2,\infty}$, the following relation holds:
\begin{equation}\label{M_2}
    \|A\|_{2,\infty}\leq \|A\|_{2} \text{ for all } A\in \mathbb{R}^{m\times n}.
\end{equation}

Next, we present some definitions and properties based on the cone $K\in \mathscr{P}(\mathbb{R}^{m})$.\\

Let $y, z \in \mathbb{R}^m$. The vector $z$ is said to be dominated by $y$ with respect to the cone $K$, denoted by $y \preceq_K z$, if $z - y \in K$. Moreover, if $z - y \in \operatorname{int}(K)$, then $z$ is said to be strictly dominated by $y$ with respect to $K$, and this relation is denoted by $y \prec_K z$.\\


The dual cone of \( K \) is given by:
$$
K^{*} := \left\{w \in \mathbb{R}^m \middle | \Tilde{y}^\top w \geq 0 \quad \forall \,\Tilde{y} \in K\right\}.
$$
As the cone $K$ is convex and closed, it induces the following relations:
\begin{align*}
  &  K^{**}=K,
   -K = \left\{\Tilde{y} \in \mathbb{R}^m ~\middle |~ w^\top \Tilde{y} \leq 0 \quad \forall\, w \in K^*\right\} \text{ and}\\
  & -\mathrm{int}(K) = \left\{\Tilde{y} \in \mathbb{R}^m ~\middle |~ w^\top \Tilde{y} \leq 0 \quad \forall\, w \in K^* \setminus \{0\}\right\}.\\
\end{align*}

Let \( C \subset K^{*} \setminus \{0\} \) be a compact set. We say that \( C \) is a generator of \( K^{*} \) if the cone of the convex hull of \( C \) generates the set \( K^{*} \), i.e.,
    $K^{*}=\mathrm{cone}(\mathrm{conv}(C))$.\\

\noindent
Note that if the given cone \( K \) is polyhedral, a generator $C$ of its dual cone \( K^{*} \) can be represented by the finite set of extremal rays of \( K^{*} \) since \( K^{*} \) is also polyhedral in this case. Moreover, for $K=\mathbb{R}^{m+}$, we can take $C$ as the canonical basis of $\mathbb{R}^{m}$ as $K^{*}=\mathbb{R}^{m+}=K$. \\

\noindent
\emph{It is important to emphasize that at no point in the analysis of this work do we impose the condition that \( K \) is finitely generated}. For this, we consider the following generator $C$ of $K^*$.\\

\begin{lemma}\emph{\cite{Ghosh CGM}}\label{com_of_C} 
    The set
\begin{equation}\label{gen_polar_cone}
    C:=\left\{w\in K^{*} ~\middle|~ w^{\top}e=1\right\}, \quad e\in \mathrm{int}(K)
\end{equation}
is a generator of the dual cone $K^*$.
\end{lemma}

\bigskip

In the following definition, we present the concepts of a minimal set concerning the ordering \( \preceq_{K} \).\\

\begin{definition}\emph{\cite{steepest2021set_optimization}} Given a set $B\in \mathscr{P}(\mathbb{R}^{m})$ and the cone $K$. Then, corresponding to the cone $K$, the following set
\begin{equation}\label{minimal_element}
  \textnormal{Min}(B,K) := \left\{ \Tilde{y} \in B \mid   K \cap(\Tilde{y}-K) =\{\Tilde{y}\} \right \}  
\end{equation}
is said to be the set of minimal elements of \( B \).\\

\end{definition}

We have a relationship between a set and its minimal elements, frequently employed in this work. This relationship is outlined in the following proposition.\\

\begin{proposition}\emph{\cite{Jahn}}\label{mini_k}
    Given a compact set $B\in \mathscr{P}(\mathbb{R}^{m})$ and the cone $K$. Then,
\[
K+\emph{\text{Min}}(B,K)=K+B.
\]  
\end{proposition}

We use the following Drummond-Svaiter functional $\phi : \mathbb{R}^m \to \mathbb{R}$ for the purpose of scalarizing vectors:
\begin{equation}\label{Dr_sc_fun}
  \phi(y):=\sup\left\{w^{\top}y ~\mid~w\in C\right\}.  
\end{equation}
Note from Lemma \ref{com_of_C} that $C$ is compact. Therefore, the functional $\phi$ is well-defined. This scalarizing functional plays an essential role in further analysis of this work. Some useful properties of the functional \( \phi \) are provided in the following lemma.\\

\begin{lemma}\emph{\cite{drummond2005steepest}}\label{scal_func}
    Let \( y \) and \( \Tilde{y} \) belong to \( \mathbb{R}^m \). Then, we have the following results: 
    \begin{enumerate}[(i)]
        \item $\phi(y+\Tilde{y})\leq \phi(y)+\phi(\Tilde{y}) $ and $\phi(y)-\phi(\Tilde{y})\leq \phi(y-\Tilde{y})$.
        \item If $y\preceq_K \Tilde{y}$, then $\phi(y)\leq \phi(\Tilde{y})$, and if $y\prec_K \Tilde{y}$, then $\phi(y)< \phi(\Tilde{y})$.
        \item $\lvert\phi(y)-\phi(\Tilde{y})\rvert\leq \|y-\Tilde{y}\|$.
        \item $-K=\left\{\Tilde{y}\in \mathbb{R}^{m} ~\middle|~ \phi(\Tilde{y})\leq 0\right\} \text{ and } -\mathrm{int}(K)=\left\{\Tilde{y}\in \mathbb{R}^{m} ~\middle|~ \phi(\Tilde{y})< 0\right\}. $
    \end{enumerate}
\end{lemma}

\bigskip


     

Before discussing the optimality conditions of the $\mathcal{SOP}_{\ell}$ \eqref{SPL}, we need to introduce some necessary tools---a few index-related set-valued maps.\\

\begin{definition}\emph{\cite{steepest2021set_optimization}}{~}Let $\Tilde{x}\in \mathbb{R}^n$.
    \begin{enumerate}[(i)]
        \item The active index set at $\Tilde{x}$ of $\mathcal{SOP}_{\ell}$ \eqref{SPL} is given by 
        \[
         I(\Tilde{x}) := \{ j \in [p] \mid f^j(\Tilde{x}) \in \textnormal{Min}(F(\Tilde{x}), K) \}.\]

         \item For a vector \( \Tilde{y} \in \mathbb{R}^m \), 
          \[
        I_{\Tilde{y}}(\Tilde{x}) := \{ j \in I(\Tilde{x}) \mid f^j(\Tilde{x}) = \Tilde{y} \}.
        \]     
    \end{enumerate}
\end{definition}

It is to be noted that, for all \( x \in \mathbb{R}^n \), we have
\[
I(x) = \bigcup_{\Tilde{y} \in \text{Min}(F(x), K)} I_{\Tilde{y}}(x) 
\]
and \( I_{\Tilde{y}}(x) = \emptyset \) whenever \( \Tilde{y} \notin \text{Min}(F(x), K) \).\\

\begin{definition}\emph{\cite{steepest2021set_optimization}}
    The map \( \omega: \mathbb{R}^n \to \mathbb{R} \) is defined as the cardinality of the set of minimal elements of \( F \), i.e., 
    \begin{equation}\label{card_fun}
    \omega(x) := |\textnormal{Min}(F(x), K)|.    
    \end{equation}
Here, \( |\cdot| \) represents the cardinality of a set.
\end{definition}
\bigskip
\begin{definition}\emph{\cite{steepest2021set_optimization}}
   At a given point $x \in \mathbb{R}^n$, let the minimal set $\mathrm{Min}(F(x), K)$ have the enumeration
$\{v_1^{x}, v_2^{x}, \dots, v_{\omega(x)}^{x}\}$,
where $\omega(x)$ denotes the number of minimal elements at $x$. Then, for the point \( x \), the partition set is defined by \begin{equation}\label{partition_set}
      P_x := \prod_{j=1}^{\omega(x)}{I_{v^{x}_{j}}(x)}.  
    \end{equation}
\end{definition}

In the following, we discuss the concept of stationarity of the $\mathcal{SOP}_{\ell}$ \eqref{SPL}.\\
\begin{definition}\emph{\cite{kumar2024nonlinear}}\label{Stationary_point}
    Let $\bar{x}\in \mathbb{R}^{n}$. We say that $\bar{x}$ is a stationary point of the $\mathcal{SOP}_{\ell}$ \eqref{SPL}, if for each $(a_1,a_2,\ldots,a_{\omega(\bar{x})})=:a\in P_{\bar{x}}$ and $d\in \mathbb{R}^{n}$, there is a $j\in [\omega(\bar{x})]$ for which $\nabla f^{a_{j}}(\bar{x})^{\top}d\notin -\mathrm{int}(K)$.
\end{definition}
\bigskip
Under the consideration of $\mathcal{SOP}_{\ell}$ \eqref{SPL}, here, we use the stationary condition as a criterion for optimality. The reasoning behind this choice is that identifying a weakly minimal point is significantly challenging. However, stationary points are weaker than weakly minimal points as every weakly minimal point is also a stationary point (see \cite[Proposition 3.1]{steepest2021set_optimization}). Thus, although our primary objective is to find a weakly minimal point of the $\mathcal{SOP}_{\ell}$ \eqref{SPL}, we ultimately focus on identifying a stationary point.\\

It is worth noting that identifying a stationary point, as defined in Definition \ref{Stationary_point}, can be quite challenging. Therefore, we introduce some functions that facilitate the identification of stationary points more easily and play a crucial role in further analysis.\\

For all $x\in \mathbb{R}^{n}$, the function $\mathcal{V}_{x}: P_{x}\times \mathbb{R}^{n}\rightarrow \mathbb{R}$ is defined as follows:
\begin{equation}\label{Vxd}
    \mathcal{V}_{x}(a,d):=\underset{j\in [\omega(x)]}{\max}\phi(\nabla f^{a_{j}}(x)^{\top}d)+\tfrac{1}{2}\|d\|^{2},
\end{equation}
where $a\in P_{x}$ and $d\in \mathbb{R}^{n}$. 
Due to the strong convexity of the function $d \mapsto \mathcal{V}_x(a,d)$ for every $x \in \mathbb{R}^n$ and $a \in P_x$,
the minimiser with respect to $d$ is unique over $\mathbb{R}^n$.
Since the partition set $P_x$ is finite, the minimum of $\mathcal{V}_x(a,d)$ over $P_x \times \mathbb{R}^n$ is guaranteed to exist.
Therefore, the scalar-valued function $\varphi : \mathbb{R}^n \to \mathbb{R}$, defined by
\begin{equation}\label{varphi}
  \varphi(x) := \min_{(a,d) \in P_x \times \mathbb{R}^n} \mathcal{V}_x(a,d),
\end{equation}
is well-defined.
\\ 

To simplify the presentation, the following notations will be used throughout:
\begin{framed}
\noindent Consider two types of points \( x \in \mathbb{R}^{n} \)---a reference point \( \bar{x} \) and an iterative point \( x_k \). Then, the following conventions are adopted:
\begin{itemize}
    \item The cardinality function value \( \omega(x) \) corresponding to \( \bar{x} \) and \( x_k \) is denoted by \( \bar{\omega} \) and \( \omega_k \), respectively.

    \item The minimization problem \( \underset{P_{x}\times \mathbb{R}^{n}}{\min}\varphi (x) \) yields solutions denoted by \( (\bar{a}, \bar{u}) \) for \( \bar{x} \) and by \( (a_k, u_k) \) for \( x_k \). Specifically, we have \( \mathcal{V}_{\bar{x}}(\bar{a}, \bar{u}) = \varphi(\bar{x}) \) and \( \mathcal{V}_{x_k}(a_k, u_k) = \varphi(x_k) \).

\end{itemize}   
\end{framed}

\bigskip

With the help of the functions defined in \eqref{Vxd} and \eqref{varphi}, an alternative tool for identifying a stationary point of the $\mathcal{SOP}_{\ell}$ \eqref{SPL}, based on Definition \ref{Stationary_point}, is provided below as a necessary and sufficient condition for stationarity.\\

\begin{proposition}\emph{\cite{kumar2024nonlinear}}
    Let $\bar{x}\in 
    \mathbb{R}^{n}$. Then, the following results are true:
    \begin{enumerate}[(i)]
        \item If $\bar{x}$ is a stationary point of the $\mathcal{SOP}_{\ell}$ \eqref{SPL}, $\varphi(\bar{x})=0$ and $\bar{u}=0$.
        \item If $\bar{x}$ is a nonstationary point of the $\mathcal{SOP}_{\ell}$ \eqref{SPL}, $\varphi(\bar{x})< 0$ and $\bar{u}\neq 0$.
    \end{enumerate}
    
\end{proposition}

\bigskip

Next, we present the concept of $K$-descent direction for the set-valued objective map $F$ of the $\mathcal{SOP}_{\ell}$ \eqref{SPL}.\\

\begin{definition}\emph{\cite{kumar2024nonlinear}}
    Let \( \bar{x} \in \mathbb{R}^n \) be a nonstationary point of the $\mathcal{SOP}_{\ell}$ \eqref{SPL}. We say that \( \bar{d} \in \mathbb{R}^n \) is a \( K \)-descent direction at \( \bar{x} \) of the $\mathcal{SOP}_{\ell}$ \eqref{SPL} if the following condition holds:
\[
\underset{j \in [\omega(\bar{x})]}{\max}  \phi \left( \nabla f^{\bar{a}_j}(\bar{x})^{\top} \bar{d} \right)  < 0.
\]
Moreover, \( \bar{d} \) is said to satisfy the sufficient descent condition at \( \bar{x} \) if there exists \( c > 0 \) such that the following relation holds:
\begin{equation}\label{S_d_c}
   \underset{j \in [\omega(\bar{x})]}{\max}  \phi \left( \nabla f^{\bar{a}_j}(\bar{x})^{\top} \bar{d} \right)  < c \underset{j \in [\omega(\bar{x})]}{\max} \left\{ \phi \left( \nabla f^{\bar{a}_j}(\bar{x})^{\top} \bar{u} \right) \right\}. 
\end{equation}
\end{definition}


We conclude this section with the following lemma that will be useful later in this paper.\\

\begin{lemma}\emph{\cite{extension2020Hager-Zhang}}
    Let $z$ and $\bar{z}$ be two unit vectors and $\delta>0$. Then, the following relation is true: 
    \begin{equation}\label{Ine_lemma}
        \|z-\bar{z}\|\leq 2\|z-\delta\bar{z}\|.
    \end{equation}
\end{lemma}

\section{Wolfe line searches and Hager-Zhang method}\label{sec_3}

In \cite{kumar2024nonlinear}, the Wolfe line searches conditions are derived for the $\mathcal{SOP}_{\ell}$ \eqref{SPL}, using the Drummond-Svaiter scalarization function \(\phi\) as given in \eqref{Dr_sc_fun}. Subsequently, the existence of a suitable step length has been established at a nonstationary point in a \(K\)-descent direction, assuming that the generator set of $K^{*}$ is finite. In order to extend this result to the case where \(K\) can be infinitely generated, we slightly modify these conditions as follows. \\

\begin{definition}
    Consider the $\mathcal{SOP}_{\ell}$ \eqref{SPL}. Let $\bar{d}$ be a $K$-descent direction at $\bar{x}$. Assume $e \in \mathrm{int}(K)$ and $0<\rho<\sigma<1$. Then, the standard Wolfe conditions are defined by 
\begin{subequations}\label{Arm with stan. wol.}
    \begin{align}
       & F(\bar{x}+\alpha \bar{d}\,)\preceq^{{\ell}}\left\{f^{\bar{a}_{j}}(\bar{x})+\rho\, \alpha \,\underset{j\in [\bar{\omega}]}{\max}\,\phi\left(\nabla f^{\bar{a}_{j}}(\bar{x})^{\top}\bar{d}\,\right) e\right\}_{j\in [\bar{\omega}]}\prec^{{\ell}}F(\bar{x}) \label{Arm in Stand.} \\
       \noalign{\noindent \text{and}}
       & \underset{j\in [\bar{\omega}]}{\max}\,\phi\left(\nabla f^{\bar{a}_{j}}(\bar{x}+\alpha\bar{d}\,)^{\top}\bar{d}\,\right)  \geq \sigma\,\underset{j\in [\bar{\omega}]}{\max}\,\phi\left(\nabla f^{\bar{a}_{j}}(\bar{x})^{\top}\bar{d}\,\right).\label{stan, wol.}
    \end{align}
\end{subequations}
Similarly, the strong Wolfe conditions are given by 
\begin{subequations}\label{Arm with strong. wol.}
    \begin{align}
       & F(\bar{x}+\alpha \bar{d}\,)\preceq^{{\ell}}\left\{f^{\bar{a}_{j}}(\bar{x})+\rho\, \alpha \,\underset{j\in [\bar{\omega}]}{\max}\,\phi\left(\nabla f^{\bar{a}_{j}}(\bar{x})^{\top}\bar{d}\,\right) e\right\}_{j\in [\bar{\omega}]}\prec^{{\ell}}F(\bar{x})  \label{Armi_con}  \\
       \noalign{\noindent \text{and}}
       & \left\lvert \underset{j\in [\bar{\omega}]}{\max}\,\phi\left(\nabla f^{\bar{a}_{j}}(\bar{x}+\alpha\bar{d}\,)^{\top}\bar{d}\,\right)\right\rvert  \leq \sigma\,\left \lvert \underset{j\in [\bar{\omega}]}{\max}\,\phi\left(\nabla f^{\bar{a}_{j}}(\bar{x})^{\top}\bar{d}\,\right)\right\rvert.\label{strong. wol.}
    \end{align}
\end{subequations}
\end{definition}

\bigskip

  In the following theorem, we present the existence of a step length from a nonstationary point in a $K$-descent direction satisfying the strong Wolfe conditions \eqref{Arm with strong. wol.} for the $\mathcal{SOP}_{\ell}$ \eqref{SPL}.\\

  \begin{theorem}\label{Exist_ST_L}
      Consider the $\mathcal{SOP}_{\ell}$ \eqref{SPL}. Let $\bar{x}$ be a nonstationary point and $\bar{d}$ be a $K$-descent direction at $\bar{x}$. Assume that there exists a bounded set $\mathcal{B}\in \mathscr{P}(\mathbb{R}^{m})$ such that $\mathcal{B}\preceq^{\ell}F(x)$ for all $x\in \mathbb{R}^{m}$, $e\in \mathrm{int}(K)$, $\sigma \in (0,1)$, and $\rho\in(\sigma,1)$. Then, there exists an interval \( I \subset \mathbb{R}_{++} \) for which each step size \( \alpha \in I \) fulfils the strong Wolfe conditions \eqref{Arm with strong. wol.}.
  \end{theorem}

  \begin{proof}
      We start by establishing the existence of an interval $\bar{I}:=[0,\bar{\alpha}], \bar{\alpha}>0,$ such that the relation \eqref{Armi_con} holds.\\ \\
      Assume that for all $\alpha \in \bar{I}$ the following is true
      \begin{equation}\label{LT1}
    f^{\bar{a}_{j}}(\bar{x}+\alpha \bar{d}\,)\preceq_{K} f^{\bar{a}_{j}}(\bar{x})+\rho\, \alpha \,\underset{j\in [\bar{\omega}]}{\max}\,\phi\left(\nabla f^{\bar{a}_{j}}(\bar{x})^{\top}\bar{d}\,\right) e \quad \text{for all } j\in [\bar{\omega}].
      \end{equation}  
      From Proposition \ref{mini_k} and the definition of a $K$-descent direction, it follows for all $\alpha\in \bar{I}$ that 
      \begin{align*}
          F(\bar{x}) ~&~ \subseteq  \left\{f^{\bar{a}_{j}}(\bar{x})\right\}_{j\in [\bar{\omega}]}+K\\
          ~&~\subseteq \left\{f^{\bar{a}_{j}}(\bar{x})+\rho\, \alpha \,\underset{j\in [\bar{\omega}]}{\max}\,\phi\left(\nabla f^{\bar{a}_{j}}(\bar{x})^{\top}\bar{d}\,\right) e\right\}_{j\in [\bar{\omega}]}+\mathrm{int}(K) +K\\
         ~&~= \left\{f^{\bar{a}_{j}}(\bar{x})+\rho\, \alpha \,\underset{j\in [\bar{\omega}]}{\max}\,\phi\left(\nabla f^{\bar{a}_{j}}(\bar{x})^{\top}\bar{d}\,\right) e\right\}_{j\in [\bar{\omega}]}+\mathrm{int}(K)
          \end{align*}
          Therefore, from the relation \eqref{LT1} and by the definition of minimal set, we have the following
          \begin{align*}
          F(\bar{x}) ~&~ \subseteq  \left\{f^{\bar{a}_{j}}(\bar{x}+\alpha \bar{d}\,)\right\}_{j\in [\bar{\omega}]}+\mathrm{int}(K)
          =  F(\bar{x}+\alpha \bar{d}\,)+\mathrm{int}(K).
          \end{align*}
          Thus, we get
          \[F(\bar{x}+\alpha \bar{d}\,)\preceq^{{\ell}}\left\{f^{\bar{a}_{j}}(\bar{x})+\rho\, \alpha \,\underset{j\in [\bar{\omega}]}{\max}\,\phi\left(\nabla f^{\bar{a}_{j}}(\bar{x})^{\top}\bar{d}\,\right) e\right\}_{j\in [\bar{\omega}]}\prec^{{\ell}}F(\bar{x})~ \text{ for all }\alpha \in \bar{I}.\]
Therefore, to establish the condition \eqref{Armi_con} for all $\alpha \in \bar{I}$, it suffices to show that the relation \eqref{LT1} is true for the same $\alpha$.\\ \\ 
If possible, assume that there exists an element $\bar{j}\in [\bar{\omega}]$ for which the relation \eqref{LT1} does not hold.
Thus, there exists a sequence $\{\alpha_{k}\} \subset \mathbb{R}_{++}$ such that $\alpha_k \rightarrow 0$ when $k\to \infty$ and
\[f^{\bar{a}_{\bar{j}}} (\bar{x}+\alpha_{k} \bar{d}\,)-f^{\bar{a}_{\bar{j}}}(\bar{x})-\rho\, \alpha_{k} ~\underset{j\in [\bar{\omega}]}{\max}\,\phi\left(\nabla f^{\bar{a}_{j}}(\bar{x})^{\top}\bar{d}\,\right)e\notin -K.\]
Accordingly, it follows that
\begin{align}
   ~&~ \nabla f^{\bar{a}_{\bar{j}}}(\bar{x})^{\top}\bar{d} -\rho ~\underset{j\in [\bar{\omega}]}{\max}\,\phi\left(\nabla f^{\bar{a}_{j}}(\bar{x})^{\top}\bar{d}\,\right) e\notin -K ~(\text{as }  k\rightarrow \infty) \nonumber\\
    \implies ~&~ \nabla f^{\bar{a}_{\bar{j}}}(\bar{x})^{\top}\bar{d} -\rho ~\underset{j\in [\bar{\omega}]}{\max}\,\phi\left(\nabla f^{\bar{a}_{j}}(\bar{x})^{\top}\bar{d}\,\right) e\notin - \mathrm{int}(K). \label{LT2}
\end{align}
From Lemma \ref{scal_func} (i) and the fact that $0<\rho<1$ and $\bar{d}$ is a $K$-descent direction, we have the following
\begin{align*}
    ~&~ \phi\left(\nabla f^{\bar{a}_{\bar{j}}}(\bar{x})^{\top}\bar{d} -\rho ~\underset{j\in [\bar{\omega}]}{\max}\,\phi\left(\nabla f^{\bar{a}_{j}}(\bar{x})^{\top}\bar{d}\,\right) e\right)
    \leq \phi\left(\nabla f^{\bar{a}_{\bar{j}}}(\bar{x})^{\top}\bar{d}\,\right) -\rho ~\underset{j\in [\bar{\omega}]}{\max}\,\phi\left(\nabla f^{\bar{a}_{j}}(\bar{x})^{\top}\bar{d}\,\right) \phi(e)  \\
    ~&~  \leq   (1-\rho) \underset{j\in [\bar{\omega}]}{\max}\,\phi\left(\nabla f^{\bar{a}_{j}}(\bar{x})^{\top}\bar{d}\,\right)<0
     \end{align*}
     because  $\phi(e)=1$. Therefore, we get
   \[\nabla f^{\bar{a}_{\bar{j}}}(\bar{x})^{\top}\bar{d} -\rho ~\underset{j\in [\bar{\omega}]}{\max}\,\phi\left(\nabla f^{\bar{a}_{j}}(\bar{x})^{\top}\bar{d}\,\right) e\in -\mathrm{int}(K),\] 
which is a contradiction to the relation \eqref{LT2}. Therefore, there exists a non-trivial interval $\bar{I}\subset \mathbb{R}_{+}$ such that the relation \eqref{Armi_con} holds.\\

We now demonstrate the existence of an interval $I\subseteq \bar{I}\neq \{0\}$ such that the condition \eqref{strong. wol.} is satisfied for all elements of I.\\

\noindent
Note that $F$ is bounded below by $\mathcal{B}$. Therefore, the relation \eqref{LT1} cannot hold for all $\alpha \in [0,\infty)$. Hence, a maximal step size \( \hat{\alpha} \geq \bar{\alpha} \) exists for which the relation \eqref{LT1} holds on the interval \( \hat{I} := [0, \hat{\alpha}] \) because $K$ is closed and the function $$\alpha \longmapsto ~f^{\bar{a}_{j}}(\bar{x}+\alpha_{k} \bar{d}\,)- f^{\bar{a}_{\hat{j}}}(\bar{x})-\rho\, \alpha_{k} ~\underset{j\in [\bar{\omega}]}{\max}\,\phi\left(\nabla f^{\bar{a}_{j}}(\bar{x})^{\top}\bar{d}\,\right)e$$ is continuous for all $j\in [\bar{\omega}]$.\\

\noindent
Since $[\bar{\omega}]$ is finite and $\hat{\alpha}$ is a maximal element of $\hat{I}$ for which the relation \eqref{LT1} can hold, therefore there exists a sequence $\{\alpha_{k}\}$, $\alpha_{k}\in \left(\hat{\alpha},\hat{\alpha}+\tfrac{1}{k}\right]$, and a $\hat{j}\in [\bar{\omega}]$ for which 
\[f^{\bar{a}_{\hat{j}}}(\bar{x}+\alpha_{k} \bar{d}\,)\npreceq_{K} f^{\bar{a}_{\hat{j}}}(\bar{x})+\rho\, \alpha_{k} ~\underset{j\in [\bar{\omega}]}{\max}\,\phi\left(\nabla f^{\bar{a}_{j}}(\bar{x})^{\top}\bar{d}\,\right) e \quad \text{for all }k.\]
Accordingly, corresponding to each $\alpha_{k}$, there exists $w_{k} \in C$ such that 
\[w^{\top}_{k}\left\{f^{\bar{a}_{\hat{j}}}(\bar{x}+\alpha_{k} \bar{d}\,)- f^{\bar{a}_{\hat{j}}}(\bar{x})-\rho\, \alpha_{k} ~\underset{j\in [\bar{\omega}]}{\max}\,\phi\left(\nabla f^{\bar{a}_{j}}(\bar{x})^{\top}\bar{d}\,\right) e\right\}>0 \quad \text{for all }k.\] 
Moreover, when $k\rightarrow \infty$, $\alpha_{k}\rightarrow \hat{\alpha}$. Therefore, there exists a subsequence $\{w_{n_k}\}$ of $\{w_{k}\}$ such that $w_{n_k}\rightarrow \hat{w} \in C$ because $C$ is compact. Consequently, we obtain
\begin{align*}
   & \hat{w}^{\top}\left\{f^{\bar{a}_{\hat{j}}}(\bar{x}+\hat{\alpha} \bar{d}\,)- f^{\bar{a}_{\hat{j}}}(\bar{x})-\rho\, \hat{\alpha} ~\underset{j\in [\bar{\omega}]}{\max}\,\phi\left(\nabla f^{\bar{a}_{j}}(\bar{x})^{\top}\bar{d}\,\right) e\right\}=0 \\
  \implies & \hat{w}^{\top}\left\{f^{\bar{a}_{\hat{j}}}(\bar{x}+\hat{\alpha} \bar{d}\,)- f^{\bar{a}_{\hat{j}}}(\bar{x})\right\}-\rho\, \hat{\alpha} ~\underset{j\in [\bar{\omega}]}{\max}\,\phi\left(\nabla f^{\bar{a}_{j}}(\bar{x})^{\top}\bar{d}\,\right)=0 
\end{align*}
because $\hat{w}\in C$. Define two functions on $\mathbb{R}_+$:  
\[\Theta_{\hat{w}}(\alpha):=\hat{w}^{\top}f^{\bar{a}_{\hat{j}}}(\bar{x}+\hat{\alpha} \bar{d}\,) \quad \text{and}\quad \psi(\alpha):=\Theta_{\hat{w}}(\alpha)-\Theta_{\hat{w}}(0)-\rho\, \alpha \,\underset{j\in [\bar{\omega}]}{\max}\,\phi\left(\nabla f^{\bar{a}_{j}}(\bar{x})^{\top}\bar{d}\,\right).\]
Note the $\psi$ is a continuous differentiable function with $\psi(\hat{\alpha})=\psi(0)$. Therefore, there exists an $\Tilde{\alpha}\in \hat{I}$ such that
\begin{align*}
 &\psi^{'}(\Tilde{\alpha})=0\\
 \implies~&~\hat{w}^{\top}\nabla f^{\bar{a}_{\hat{j}}}(\bar{x}+\Tilde{\alpha}\bar{d}\,)^{\top}\bar{d}-\rho  ~\underset{j\in [\bar{\omega}]}{\max}\,\phi\left(\nabla f^{\bar{a}_{j}}(\bar{x})^{\top}\bar{d}\,\right)=0\\
 \implies ~&~ \phi\left(\nabla f^{\bar{a}_{\hat{j}}}(\bar{x}+\Tilde{\alpha}\bar{d}\,)^{\top}\bar{d}\,\right)-\rho  ~\underset{j\in [\bar{\omega}]}{\max}\,\phi\left(\nabla f^{\bar{a}_{j}}(\bar{x})^{\top}\bar{d}\,\right)\geq 0\\
    \implies ~&~\rho\,  \underset{j\in [\bar{\omega}]}{\max}\,\phi\left(\nabla f^{\bar{a}_{j}}(\bar{x})^{\top}\bar{d}\,\right)\leq \underset{j\in [\bar{\omega}]}{\max}\,\phi\left(\nabla f^{\bar{a}_{j}}(\bar{x}+\Tilde{\alpha}\bar{d}\,)^{\top}\bar{d}\,\right). 
\end{align*}
Therefore, there exists $\alpha^{*}\in (0,\Tilde{\alpha}]\subseteq \hat{I}$ such that 
\begin{align}
    ~&~\rho\,  \underset{j\in [\bar{\omega}]}{\max}\,\phi\left(\nabla f^{\bar{a}_{j}}(\bar{x})^{\top}\bar{d}\,\right)= \underset{j\in [\bar{\omega}]}{\max}\,\phi\left(\nabla f^{\bar{a}_{j}}(\bar{x}+\alpha^{*}\bar{d}\,)^{\top}\bar{d}\,\right)\nonumber\\
    \implies~&~\sigma\, \underset{j\in [\bar{\omega}]}{\max}\,\phi\left(\nabla f^{\bar{a}_{j}}(\bar{x})^{\top}\bar{d}\,\right)< \underset{j\in [\bar{\omega}]}{\max}\,\phi\left(\nabla f^{\bar{a}_{j}}(\bar{x}+\alpha^{*}\bar{d}\,)^{\top}\bar{d}\,\right)\label{LT3}
\end{align}
because $\sigma\in (\rho,1)$ and $\underset{j\in [\bar{\omega}]}{\max}\,\phi\left(\nabla f^{\bar{a}_{j}}(\bar{x})^{\top}\bar{d}\,\right)<0$.\\
Thus, there exists a neighbourhood $I\subseteq \hat{I}\setminus \{0\}$ for which the relation \eqref{LT3} holds. Thus, the proof is completed.     
  \end{proof}

  \begin{remark}
It is important to remark that at any nonstationary point of \( \mathcal{SOP}_{\ell} \) \eqref{SPL}, when a \(K\)-descent direction is taken at that point, there exists an interval \( I \subset \mathbb{R}_{++} \) such that every step length \( \alpha \in I \) fulfilling the strong Wolfe conditions \eqref{Arm with strong. wol.} necessarily fulfils the standard Wolfe conditions \eqref{Arm with stan. wol.} as well.
 \end{remark}

\bigskip

Now, we focus on the development of a Hazer-Zhang conjugate gradient method for solving the $\mathcal{SOP}_{\ell}$ \eqref{SPL}. The general form of the conjugate gradient method for this problem at iteration $k$ is given by   
\begin{equation*}
  x_{k+1} := x_k + \alpha_k d_k,  
\end{equation*}
where \( \alpha_k \) is the step length and \( d_k \) is the search direction defined by
\begin{equation}\label{cgd}
 d_{k}=\left\{\begin{aligned}
   &u_{k} &&\text{ if } k = 0\\
   &u_{k}+\beta_{k}d_{k-1} && \text{ otherwise},
 \end{aligned}\right.   
\end{equation}
where $\beta_{k}$ is a scalar conjugate parameter.\\

Hazer-Zhang scalar conjugate parameter of the $\mathcal{SOP}_{\ell}$ \eqref{SPL} is defined as follows: 

 \begin{equation}\label{BHZ}
    \beta^{\mathrm{HZ}}_k:=\frac{1}{\mathcal{H}_2-\mathcal{H}_3}\left\{\mathcal{H}_{1}-\frac{\mu \Tilde{\mathcal{H}}_2}{\mathcal{H}_2-\mathcal{H}_3}\left\|\nabla f^{a_{j_{k-1}}}(x_{k}) -\nabla f^{a_{j_{k-1}}}(x_{k-1})\right \|^{2}_{2,\infty}\right\},
\end{equation}
where \begin{align*}
    &\mathcal{H}_1:=\underset{j\in [\omega_k]}{\max}~\phi \left(\nabla f^{a_{k,j}}(x_{k-1})^{\top}u_{k}\right)-\underset{j\in [\omega_k]}{\max}~ \phi \left(\nabla f^{a_{k,j}}(x_{k})^{\top}u_{k}\right),\\ &\mathcal{H}_2:=\underset{j\in [\omega_{k-1}]}{\max} ~\phi \left(\nabla f^{a_{k-1,j}}(x_{k})^{\top}d_{k-1}\right),\\
    & \bar{\mathcal{H}}_2:=\underset{j\in [\omega_{k}]}{\max}\,\phi \left(\nabla f^{a_{k,j}}(x_{k})^{\top}d_{k-1}\right),\\
 & \Tilde{\mathcal{H}}_2:=\max\{\mathcal{H}_2,\bar{\mathcal{H}}_2\}, \\ 
    &\mathcal{H}_3:=\underset{j\in [\omega_{k-1}]}{\max}\,\phi \left(\nabla f^{a_{k-1,j}}(x_{k-1})^{\top}d_{k-1}\right),  \\
    & j_{k-1}:=\underset{j\in [\omega_k]}{\text{argmax}}~\phi \left(\nabla f^{a_{k,j}}(x_{k-1})^{\top}u_{k}\right),~ \text{ and } \mu>\tfrac{1}{2}.
\end{align*}
Note that if $p=1$, the $\mathcal{SOP}_{\ell}$ \eqref{SPL} reduces into a vector optimization problem. As discussed in \cite{hu2024alternative}, we select the following $\beta_{k}$ by restricting the nonnegativity of $\beta^{\mathrm{HZ}}_k$ in the proposed method: 
\begin{equation}\label{beta_k}
    \beta_{k}:=\max\left\{0,\beta^{\mathrm{HZ}}_k\right\}.
\end{equation}
The proposed method is outlined in the following algorithm (Algorithm \ref{algo}).\\

 \begin{algorithm}[!htb] 
\caption{Non-linear Hager-Zhang conjugate gradient method for the $\mathcal{SOP}_{\ell}$ \eqref{SPL} }\label{algo}
\begin{enumerate}[\text{Step} 1]
\item  Initialize the constants $\rho\in (0,1),~\sigma\in (\rho,1)$, and $e\in \mathrm{int}(K)$. Provide a tolerance value $\varepsilon>0$. Choose a starting point $x_{0}\in\mathbb{R}^{n}$.
Set the iteration counter $k := 0.$ \\

\item Compute $M_k:=\mathrm{Min}(F(x_k),K),~\omega_{k} := \lvert \mathrm{Min}(F(x_k),K)\rvert,\text{ and }P_{k}:=P_{x_{k}}.$\\

\item Find
$$(a_{k},u_{k})\in \underset{(a,u)\in P_{k}\times\mathbb{R}^{n}}{\mathrm{argmin}} \left(\underset{j\in [\omega_{k}]}{\max}\phi\left(\nabla f^{a_{j}}(x_{k})^{\top}u\right)+ \tfrac{1}{2}\|u\|^{2}\right).$$

\item Stop if $\|u_{k}\| < \varepsilon$. Otherwise, go to Step 5.\\

\item Find $d_{k}$ using equations \eqref{cgd} and \eqref{beta_k}.\\

\item Find a step length $\alpha_{k}>0$ satisfying the conditions \eqref{Arm with stan. wol.}\\

\item    Obtain $x_{k+1} := x_{k}+\alpha_{k}d_{k}.$ Set $ k:=  k+1,$ and go to Step 2.
   \end{enumerate}
\end{algorithm}

\noindent
\textbf{Well-definedness of Algorithm \ref{algo}}: The well-definedness of Algorithm \ref{algo} relies on Step 3 and Step 6. Specifically, it requires the existence of a minimum point for the function $\varphi$ defined in \eqref{varphi}, as well as a step length that satisfies the standard Wolfe conditions \eqref{Arm with stan. wol.}. It is important to note that for any given $x_{k}\in \mathbb{R}^{n}$ and $a\in P_{x_{k}}$, the function $\mathcal{V}{x}(a,\cdot)$ is strongly convex, and the set $[p]$ is finite. Consequently, at each point $x_{k}$, there exists a minimum point $(a_{k},u_{k})$ for the function $\varphi$.
Moreover, at any iteration $k$ that does not satisfy Step 4, later, we show that the direction $d_{k}$ is $K$-descent (Theorem \ref{P_S_D}). Accordingly, from Theorem \ref{Exist_ST_L}, there exists $\alpha_{k}$ which holds the conditions \eqref{Arm with stan. wol.}. Hence, Algorithm \ref{algo} is well-defined.

\section{Convergence analysis}\label{sec_4}
To analyse the convergence of a sequence produced by Algorithm \ref{algo}, we adopt the following standard and widely used assumptions.\\

\begin{framed}

\begin{assumption}\label{Ass. 3}
 The lower level set $\mathcal{L}:=\left\{x \in \mathbb{R}^{n}: F(x)\preceq^{\ell}F(x_{0})\right\}$ is assumed to be bounded, and accordingly there exists $\gamma\in \mathbb{R}_{+}$ such that $\|x\|\leq \gamma $ for all $x\in \mathcal{L}$.   \\
\end{assumption}


\begin{assumption}\label{lip_ass}
Let $\Gamma$ be a nonempty open neighborhood of $\mathcal{L}$. 
The Jacobian $\nabla f^{i}$, for $i \in [p]$, satisfy a uniform Lipschitz condition on $\Gamma$ with constant $L>0$, namely,
\[
\|\nabla f^{i}(y)-\nabla f^{i}(z)\| \le L\|y-z\|, \qquad \forall \,i \in [p],\;  y,z \in \Gamma.
\]
\end{assumption}

\begin{assumption}\label{bddbe_ass}
Let $\{S_k\}$, with $S_k \in F(\mathcal{L})$, be any sequence such that $S_{k+1} \preceq^{\ell} S_k$ for all $k \ge 0$. Then, there exists a set $S$ of $\mathbb{R}^{m}$ which is bounded and satisfying  $S \preceq^{\ell} S_k$ for all $k \ge 0$.

\end{assumption}

\end{framed}

\bigskip

We now focus on establishing the global convergence of the proposed Algorithm \ref{algo}. To begin, we demonstrate that the direction $d_{k}$, generated by Algorithm \ref{algo} is $K$-descent direction. This is achieved by leveraging a key relation, as presented in the following lemma.\\

\begin{lemma}\emph{\cite{hu2024alternative}}
   Suppose that \emph{Assumption \ref{lip_ass}} holds. Then, the following relation is true for all $j \in [p]$: 
   \begin{equation}\label{lip_cond}
    \lvert \phi(\nabla f^{j}(x)^{\top}d)-  \phi(\nabla f^{j}(y)^{\top}d)\rvert \leq \|d\|\|\nabla f^{j}(x)-\nabla f^{j}(y)\|_{2,\infty}.  
   \end{equation}
\end{lemma}

\bigskip

Note that the direction $d_{k}$ becomes a $K$-descent if it satisfies the sufficient descent condition \eqref{S_d_c}. The following theorem shows that the direction $d_{k}$ satisfies the sufficient descent condition \eqref{S_d_c} and hence $d_{k}$ is a $K$-descent direction.\\

\begin{theorem}\label{P_S_D}
Consider the $\mathcal{SOP}_{\ell}$ \eqref{SPL}, and assume that \emph{Assumption \ref{lip_ass}} holds. Then, the direction \( d_k \) generated by \emph{Algorithm \ref{algo}} satisfies the sufficient descent condition \eqref{S_d_c} with the constant \( c = 1 - \tfrac{1}{2\mu} \).
\end{theorem}

\begin{proof}
We prove the result by mathematical induction. For the base case, note that $d_{0}=u_{0}$, and consequently,
\[
\underset{j \in [\omega_{0}]}{\max}\phi\!\left(\nabla f^{a_{0,j}}(x_{0})^{\top} d_{0}\right)
<
\left(1-\tfrac{1}{2\mu}\right)
\underset{j \in [\omega_{0}]}{\max}\phi\!\left(\nabla f^{a_{0,j}}(x_{0})^{\top} u_{0}\right).
\]
Hence, the sufficient descent condition~\eqref{S_d_c} holds for $k=0$ with $c=1-\tfrac{1}{2\mu}$.\\

\noindent
Now assume that $d_{k-1}$ satisfies the sufficient descent condition~\eqref{S_d_c} with $c=1-\tfrac{1}{2\mu}$. In the following, it is shown that $d_k$ also satisfies the sufficient descent condition~\eqref{S_d_c} with the same constant.\\

\noindent
    From the definition of $d_{k}$, we have
    \begin{align}
        &d_{k}=u_{k}+\beta_{k}d_{k-1}\nonumber\\
        \implies &\underset{j\in [\omega_k]}{\max}\,\phi \left(\nabla f^{a_{k,j}}(x_{k})^{\top}d_{k}\right)\leq \underset{j\in [\omega_k]}{\max}\,\phi \left(\nabla f^{a_{k,j}}(x_{k})^{\top}u_{k}\right)+\beta_{k}\bar{\mathcal{H}}_2 \label{R_15}    
    \end{align}
    because of Lemma \ref{scal_func} (i).\\

    \noindent Note that if $\bar{\mathcal{H}}_2\leq 0 $ or $\beta_{k}=0$, then 
    \begin{equation}\label{suff_value}
    \underset{j \in [\omega_{k}]}{\max}  \phi \left( \nabla f^{a_j}(x_{k})^{\top} d_{k} \right)  < c \underset{j \in [\omega_{k}]}{\max}  \phi \left( \nabla f^{a_j}(x_{k})^{\top} u_{k} \right)  \text{ as }c<1.
\end{equation}
We assume that $\bar{\mathcal{H}}_2 > 0 $ and $\beta_k=\beta_k^{\mathrm{HZ}}>0$.
From the equation \eqref{BHZ}, we get\\
\begin{align}\label{P_1}
 \beta_{k}= \frac{1}{\mathcal{H}_2-\mathcal{H}_3}\left\{\mathcal{H}_{1}-\frac{\mu \Tilde{\mathcal{H}}_2}{\mathcal{H}_2-\mathcal{H}_3}\left\|\nabla f^{a_{j_{k-1}}}(x_{k}) -\nabla f^{a_{j_{k-1}}}(x_{k-1})\right \|^{2}_{2,\infty}\right\}.
\end{align} 
Notice that 
\[\mathcal{H}_1=\underset{j\in [\omega_k]}{\max}\,\phi \left(\nabla f^{a_{k,j}}(x_{k-1})^{\top}u_{k}\right)-\underset{j\in [\omega_k]}{\max}\,\phi \left\{\nabla f^{a_{k,j}}(x_{k})^{\top}u_{k}\right\}.\]
Denoting $j_{k-1}\in\underset{j\in [\omega_k]}{\text{argmax}}\,\phi \left(\nabla f^{a_{k,j}}(x_{k-1})^{\top}u_{k}\right)$. Therefore, it follows that
\begin{align}
  \mathcal{H}_1 &\leq \phi \left(\nabla f^{a_{j_{k-1}}}(x_{k-1})^{\top}u_{k}\right)-\phi \left\{\nabla f^{a_{j_{k-1}}}(x_{k})^{\top}u_{k}\right\}\nonumber\\
  &\overset{\eqref{lip_cond}}{\leq}\|u_k\| \|\nabla f^{a_{j_{k-1}}}(x_{k-1})-\nabla f^{a_{j_{k-1}}}(x_{k})\|_{2,\infty}.  
 \label{R_11}
\end{align}
Furthermore, from the standard Wolfe condition \eqref{stan, wol.} with the fact that $d_{k-1}$ is a $K$-descent direction, we obtain
\begin{align}
&\underset{j\in [\omega_{k-1}]}{\max}\,\phi \left(\nabla f^{a_{k-1,j}}(x_{k})^{\top}d_{k-1}\right)\geq \sigma \underset{j\in [\omega_{k-1}]}{\max}\,\phi \left(\nabla f^{a_{k-1,j}}(x_{k-1})^{\top}d_{k-1}\right)\nonumber\\
\implies &\underset{j\in [\omega_{k-1}]}{\max}\,\phi \left(\nabla f^{a_{k-1,j}}(x_{k})^{\top}d_{k-1}\right)\geq  \underset{j\in [\omega_{k-1}]}{\max}\,\phi \left(\nabla f^{a_{k-1,j}}(x_{k-1})^{\top}d_{k-1}\right) \nonumber\\ 
\implies & \mathcal{H}_2-\mathcal{H}_3 >0.\label{R_12}
\end{align}
Now, from the relations \eqref{P_1}, \eqref{R_11},
and \eqref{R_12}, it follows that
\begin{align*}
     \beta_{k} & = \frac{1}{\mathcal{H}_2-\mathcal{H}_3}\left\{\mathcal{H}_{1}-\frac{\mu \Tilde{\mathcal{H}}_2}{\mathcal{H}_2-\mathcal{H}_3}\left\|\nabla f^{a_{j_{k-1}}}(x_{k}) -\nabla f^{a_{j_{k-1}}}(x_{k-1})\right \|^{2}_{2,\infty}\right\}\\
     & \leq \frac{1}{\mathcal{H}_2-\mathcal{H}_3}\|u_k\| \|\nabla f^{a_{j_{k-1}}}(x_{k-1})-\nabla f^{a_{j_{k-1}}}(x_{k})\|_{2,\infty}\\
     &~~~-\frac{\mu \Tilde{\mathcal{H}}_2}{(\mathcal{H}_2-\mathcal{H}_3)^2}\left\|\nabla f^{a_{j_{k-1}}}(x_{k}) -\nabla f^{a_{j_{k-1}}}(x_{k-1})\right \|^{2}_{2,\infty}.
\end{align*}
  Since $\bar{\mathcal{H}}_2>0$, therefore we get
  \begin{align}
 \beta_{k} \bar{\mathcal{H}}_2 & \leq \frac{\bar{\mathcal{H}}_2(\mathcal{H}_2-\mathcal{H}_3)}{(\mathcal{H}_2-\mathcal{H}_3)^{2}} \|u_k\|\|\nabla f^{a_{j_{k-1}}}(x_{k-1})-\nabla f^{a_{j_{k-1}}}(x_{k})\|_{2,\infty}\nonumber\\
     &~~~-\frac{\mu (\bar{\mathcal{H}}_2)^{2}}{(\mathcal{H}_2-\mathcal{H}_3)^2}\left\|\nabla f^{a_{j_{k-1}}}(x_{k}) -\nabla f^{a_{j_{k-1}}}(x_{k-1})\right \|^{2}_{2,\infty}. \label{R_13}
\end{align}
We know that $b_{1}b_{2}\leq \tfrac{1}{2}b^{2}_{1} + \tfrac{1}{2}b^{2}_{2}$ for any real number $b_{1}$ and $b_{2}$.
Denote $$b_{1}:=\sqrt{2 \mu}\bar{\mathcal{H}}_2 \|\nabla f^{a_{j_{k-1}}}(x_{k-1})-\nabla f^{a_{j_{k-1}}}(x_{k})\|_{2,\infty} \text{ and } b_{2}:=\tfrac{1}{\sqrt{2 \mu}}(\mathcal{H}_2-\mathcal{H}_3)\|u_{k}\|.$$ Thus, the relation \eqref{R_13} becomes
\begin{align}
    \beta_{k} \bar{\mathcal{H}}_2 & \leq \frac{1}{4\mu}\|u_{k}\|^{2}+\frac{\mu (\bar{\mathcal{H}}_2)^{2}}{(\mathcal{H}_2-\mathcal{H}_3)^2}\left\|\nabla f^{a_{j_{k-1}}}(x_{k}) -\nabla f^{a_{j_{k-1}}}(x_{k-1})\right \|^{2}_{2,\infty}\nonumber\\
     &~~~-\frac{\mu (\bar{\mathcal{H}}_2)^{2}}{(\mathcal{H}_2-\mathcal{H}_3)^2}\left\|\nabla f^{a_{j_{k-1}}}(x_{k}) -\nabla f^{a_{j_{k-1}}}(x_{k-1})\right \|^{2}_{2,\infty} \nonumber \\
     \implies  \beta_{k} \bar{\mathcal{H}}_2 & \leq \frac{1}{4\mu}\|u_{k}\|^{2}. \label{R_14}
\end{align}
From the relations \eqref{R_15} and \eqref{R_14}, we get
\begin{align}
    &\underset{j\in [\omega_k]}{\max}\,\phi \left(\nabla f^{a_{k,j}}(x_{k})^{\top}d_{k}\right)\leq \underset{j\in [\omega_k]}{\max}\,\phi \left(\nabla f^{a_{k,j}}(x_{k})^{\top}u_{k}\right) + \frac{1}{4\mu}\|u_{k}\|^{2}\nonumber\\
    \implies &    \underset{j \in [\omega_{k}]}{\max}  \phi \left( \nabla f^{a_{k,j}}(x_{k})^{\top} d_{k} \right)  < (1-\tfrac{1}{2\mu}) \underset{j \in [\omega_{k}]}{\max}  \phi \left( \nabla f^{a_{k,j}}(x_{k})^{\top} u_{k} \right) \label{R_16}   
\end{align}
because $\underset{j\in [\omega_k]}{\max}\,\phi \left(\nabla f^{a_{k,j}}(x_{k})^{\top}u_{k}\right)+\tfrac{1}{2}\|u_{k}\|^{2}<0$ as $x_{k}$ is nonstationary point.\\ Hence, from \eqref{suff_value} and \eqref{R_16},  $d_{k}$ satisfies sufficient descent condition \eqref{S_d_c} with $c:=1-\tfrac{1}{2\mu}$.
\end{proof}

\bigskip 

In the next theorem, we establish a Zoutendijk-type condition. This condition will play an important role in the convergence of the proposed Algorithm \ref{algo}\\

\bigskip
\begin{theorem}\label{the. zoun.}
Suppose that \emph{Assumptions \ref{lip_ass} and \ref{bddbe_ass}} hold. Let the iterative point $x_{k}$ and search direction $d_k$ be generated by \emph{Algorithm \ref{algo}} for all $k = 0, 1, 2, \ldots$. Then,
    \begin{align}\label{Zounten.}
        \sum_{k= 0}^{\infty}\left\{{\frac{\underset{j \in [\omega_{k}]}{\max}  \phi \left( \nabla f^{a_{k,j}}(x_{k})^{\top} d_{k} \right) }{\|d_{k}\|}}\right\}^{2}< +\infty .
    \end{align}
\end{theorem}
\begin{proof}
    Note that $\alpha_{k}$ satisfies standard Wolfe condition at $x_{k}$ for the direction $d_{k}$. Therefore, from \eqref{stan, wol.}, it follows that 
    \begin{align*}
        ~&~\sigma\,\underset{j \in [\omega_{k}]}{\max}  \phi \left( \nabla f^{a_{k,j}}(x_{k})^{\top} d_{k} \right) < \underset{j \in [\omega_{k}]}{\max}  \phi \left( \nabla f^{a_{k,j}}(x_{k+1})^{\top} d_{k} \right) ,\\
        \text{i.e.,} ~&~ (\sigma-1)\underset{j \in [\omega_{k}]}{\max}  \phi \left( \nabla f^{a_{k,j}}(x_{k})^{\top} d_{k} \right) < \underset{j \in [\omega_{k}]}{\max}  \phi \left( \nabla f^{a_{k,j}}(x_{k+1})^{\top} d_{k} \right) -\underset{j \in [\omega_{k}]}{\max}  \phi \left( \nabla f^{a_{k,j}}(x_{k})^{\top} d_{k} \right) .
    \end{align*}
    Denoting $j_{k}\in\underset{j\in [\omega_{k}]}{\text{argmax}}\, \phi \left( \nabla f^{a_{k,j}}(x_{k+1})^{\top} d_{k} \right) $. Then, we get  
    \begin{align*}
        (\sigma-1)\underset{j \in [\omega_{k}]}{\max}  \phi \left( \nabla f^{a_{k,j}}(x_{k})^{\top} d_{k} \right) &<  \phi \left( \nabla f^{a_{j_{k}}}(x_{k+1})^{\top} d_{k} \right) - \phi \left( \nabla f^{a_{j_{k}}}(x_{k})^{\top} d_{k} \right)\\
        &\leq\left\|\nabla f^{a_{j_{k}}}(x_{k+1}) -\nabla f^{a_{j_{k}}}(x_{k})\right\|\|d_{k}\| \leq L\alpha_{k}\|d_{k}\|^{2}
    \end{align*}
     because of Lemma \ref{scal_func} (iii) and Assumption \ref{lip_ass}. Therefore, we have
    \begin{equation}\label{ZAT1}
       \left\{ \tfrac{\underset{j \in [\omega_{k}]}{\max}  \phi \left( \nabla f^{a_{k,j}}(x_{k})^{\top} d_{k} \right)}{\|d_{k}\|}\right\}^{2}<\tfrac{L\alpha_{k}}{\sigma-1}\underset{j \in [\omega_{k}]}{\max}  \phi \left( \nabla f^{a_{k,j}}(x_{k})^{\top} d_{k} \right) .
    \end{equation}
    For establishing the relation \eqref{Zounten.}, we now define a function $\mathcal{J}:\mathbb{R}^{m}\rightarrow \mathbb{R}\cup \{-\infty\}$ by
    \[\mathcal{J}(Z):=\underset{z\in Z}{\inf}\phi(z), \quad Z\in \mathscr{P}(\mathbb{R}^{m}).\]
    From Lemma \ref{scal_func}, the function $\mathcal{J}$ preserves monotonicity. Specifically, for $Z_{1}, Z_{2}\in \mathscr{P}(\mathbb{R}^{m})$,
    \[ Z_{1}\preceq^{\ell}Z_{2}\implies \mathcal{J}(Z_1)\leq \mathcal{J}(Z_2). \]
   From the standard Wolfe condition \eqref{Arm in Stand.}, we have 
    \[ F(x_{k+1})\preceq^{{\ell}}\left\{f^{a_{k,j}}(x_k)+\rho\, \alpha_{k} ~\underset{j\in [\omega_k]}{\max}\phi\left(\nabla f^{a_{k,j}}(x_{k})^{\top}d_{k}\right)e\right\}_{j\in [\omega_{k}]}.\]
    Therefore, 
    \begin{align*}
     \mathcal{J} \circ F(x_{k+1})&\leq \underset{j\in [\omega_{k}]}{\min}\phi\left(f^{a_{k,j}}(x_k)+\rho\, \alpha_{k} \,\underset{j\in [\omega_k]}{\max}\phi\left(\nabla f^{a_{k,j}}(x_{k})^{\top}d_{k}\right)e\right)\\
     & \leq \underset{j\in [\omega_{k}]}{\min}\phi \left(f^{a_{k,j}}(x_k)\right)-\rho\, \alpha_{k} \,\underset{j\in [\omega_k]}{\max}\left\{\phi\left(\nabla f^{a_{k,j}}(x_{k})^{\top}d_{k}\right)\right\}\phi(-e)\\
     & = \underset{j\in [\omega_{k}]}{\min}\phi \left(f^{a_{k,j}}(x_k)\right)+\rho\, \alpha_{k} \,\underset{j\in [\omega_k]}{\max}\phi\left(\nabla f^{a_{k,j}}(x_{k})^{\top}d_{k}\right)\\
     & = \mathcal{J} \circ F(x_{k})+\rho\, \alpha_{k} ~\underset{j\in [\omega_k]}{\max}\phi\left(\nabla f^{a_{k,j}}(x_{k})^{\top}d_{k}\right)\\
     \implies \mathcal{J} \circ F(x_{k+1})&-\mathcal{J} \circ F(x_{k})\leq \rho\, \alpha_{k} ~\underset{j\in [\omega_k]}{\max}\phi\left(\nabla f^{\bar{a}_{k,j}}(x_{k})^{\top}d_{k}\right)\\
     \implies \mathcal{J} \circ F(x_{k+1})&-\mathcal{J} \circ F(x_{k-1})\leq \rho\sum_{i=k-1}^{k} \alpha_{i} \,\underset{j\in [\omega_i]}{\max}\phi\left(\nabla f^{a_{i,j}}(x_{i})^{\top}d_{i}\right)\\
     &\vdots\\
     \implies \mathcal{J} \circ F(x_{k+1})&-\mathcal{J} \circ F(x_{0})\leq \rho\sum_{i=0}^{k} \alpha_{i} ~\underset{j\in [\omega_i]}{\max}\phi\left(\nabla f^{a_{i,j}}(x_{i})^{\top}d_{i}\right).
    \end{align*}
    Note that $S\preceq^{\ell}F(x_{k})$ for all $k=0,1,2,\ldots$ by Assumption \ref{bddbe_ass}. Denoting $\hat{S}:=S-\mathcal{J} \circ F(x_{0})$, we have 
    \begin{align*}
       & \hat{S}\leq \rho\sum_{i=0}^{k} \alpha_{i} ~\underset{j\in [\omega_i]}{\max}\left\{\phi\left(\nabla f^{a_{i,j}}(x_{i})^{\top}d_{i}\right)\right\}<0\\
       \implies &  \frac{\hat{S}}{\sigma-1}\geq \frac{\rho}{\sigma-1}\sum_{i=0}^{k} \alpha_{i} ~\underset{j\in [\omega_i]}{\max}\left\{\phi\left(\nabla f^{a_{i,j}}(x_{i})^{\top}d_{i}\right)\right\}>0\\
       \implies &  \frac{1}{\sigma-1}\sum_{i=0}^{\infty} \alpha_{i} ~\underset{j\in [\omega_i]}{\max}\left\{\phi\left(\nabla f^{a_{i,j}}(x_{i})^{\top}d_{i}\right)\right\}<+\infty.
    \end{align*}
    Therefore, from the relation \eqref{ZAT1}, we get 
     \begin{align*}
        \sum_{k= 0}^{\infty}\left\{{\tfrac{\underset{j \in [\omega_{k}]}{\max}  \phi \left( \nabla f^{a_{k,j}}(x_{k})^{\top} d_{k} \right) }{\|d_{k}\|}}\right\}^{2}< +\infty .
    \end{align*}
    
\end{proof}

Next, we report a result, which is important for further analysis. \\

\begin{lemma}\label{Lemma_U_B}
    Suppose that $u_{k}$ is generated by \emph{Algorithm \ref{algo}} corresponding to each $x_{k}$. Moreover, \emph{Assumption \ref{Ass. 3}} holds. Then, there exists a positive constant $u_{b}$ such that 
    \begin{equation}\label{U_B_u}
        \|u_{k}\|\leq u_{b} \text{ for all }k\geq 0.
    \end{equation}
\end{lemma}
\begin{proof}
    Consider the function $\bar{\mathcal{V}}_{x}: \mathscr{P}([p])\times \mathbb{R}^{n}\rightarrow \mathbb{R}$ given by
    \begin{equation*}
    \bar{\mathcal{V}}_{x}(\zeta,d):=\underset{j\in \zeta}{\max}\,\phi(\nabla f^{j}(x)^{\top}d)+\tfrac{1}{2}\|d\|^{2}.
\end{equation*}
Since the mapping $d\longmapsto \phi(\nabla f^{j}(x)^{\top}d)$ is continuous for all $j\in [p]$, $d\longmapsto \bar{\mathcal{V}}_{x}(\zeta,\cdot)$ is a continuous mapping for each $\zeta \in \mathscr{P}([p])$.\\

\noindent
Moreover, $\mathcal{V}_{x}(\zeta,\cdot)$ is a strongly convex function of $d$ for any $x\in \mathbb{R}^{n}$ and $\zeta \in \mathscr{P}([p])$ because $\phi$ is sublinear.\\

\noindent Accordingly, there exists a unique minimizer of $\bar{\mathcal{V}}_{x}(\zeta,\cdot)$. Let $d_{\zeta}(x):=\underset{d\in \mathbb{R}^{n}}{\operatorname{argmin}}\,\mathcal{V}_{x}(\zeta,d)$. By the continuity of $\bar{\mathcal{V}}_{x}(\zeta,\cdot)$ and the uniqueness of its minimizer, it follows that for each $\zeta \in \mathscr{P}([p])$, the mapping $d_{\zeta}$ is continuous. This argument is similar to the proof of the third part of Lemma 3.3 in \cite{drummond2005steepest}.\\

\noindent 
Note that $x_{k}$ is generated by Algorithm \ref{algo}. Therefore, $x_{k}$ belongs to $ \mathcal{L}$ for all $k\geq0$.\\

\noindent
Since Assumption \ref{Ass. 3} holds, accordingly $\mathcal{L}$ is compact. Therefore, $d_{\zeta}$ is a bounded function on the set $\mathcal{L}$.\\

\noindent
Thus, by definition of $u_{k}$, $u_{k}$ is bounded for all $x_k\in \mathcal{L}$. Therefore, there exists a positive constant $u_{b}$ such that the relation \eqref{U_B_u} holds.
\end{proof}

In order to establish the convergence, we assume that Algorithm \ref{algo} does not stop at a stationary point, i.e, the $\mathcal{SOP}_{\ell}$ \eqref{SPL} is not solved in finite iterations. Then, we show that Algorithm \ref{algo} generates a sequence $\{x_{k}\}$ such that the $\mathcal{SOP}_{\ell}$ \eqref{SPL} is solved asymptotically,  
\begin{equation}\label{asymptot}
    \text{i.e.,}\quad \underset{k\rightarrow\infty}{\liminf} \|u_{k}\|=0 .
\end{equation}

Based on equation \eqref{asymptot}, we employ a proof by contradiction to establish the global convergence of the proposed method. Specifically, we first assume that relation \eqref{asymptot} does not hold and define a recursive relation on the sequence $\{d_k\}$ by estimating the parameter $\beta_{k}$. Subsequently, we demonstrate that the derived relation on $\{d_k\}$ leads to a contradiction, which confirms that relation \eqref{asymptot} must hold. Hence, the global convergence of the proposed method is established.\\

In the following lemma, we establish an estimate for the parameter $\beta _{k}$.\\

\begin{lemma}\label{Lemma_est_Beta}
    Suppose \emph{Assumptions \ref{Ass. 3}, \ref{lip_ass}}, and \emph{\ref{bddbe_ass}} hold and there exists a positive constant $l_b$ such that for all $k\geq 0$,
    \begin{equation}\label{u_l_b}
        \|u_{k}\|\geq l_b .
    \end{equation}
\end{lemma}
\noindent
Then, for all $k$, the following relation is true
\begin{equation}\label{R_est}
    \beta_{k}\leq M \|x_{k}-x_{k-1}\|,
\end{equation}
where $M:=\frac{2}{(1-\sigma)^{2}\left(1-\frac{1}{2\mu}\right)l^{2}_b}\left\{Lu_{b}+2\mu L^{2}\gamma\max \left\{1,\frac{\sigma}{1-\sigma}\right\}\right\}.$

\begin{proof}
    For $\beta_k^{\mathrm{HZ}}\leq 0$, the inequality \eqref{R_est} is trivial. Therefore, we are assuming $\beta_k^{\mathrm{HZ}}> 0$. Accordingly, we have 
    \[\beta_{k}  = \frac{1}{\mathcal{H}_2-\mathcal{H}_3}\left\{\mathcal{H}_{1}-\frac{\mu \Tilde{\mathcal{H}}_2}{\mathcal{H}_2-\mathcal{H}_3}\left\|\nabla f^{a_{j_{k-1}}}(x_{k}) -\nabla f^{a_{j_{k-1}}}(x_{k-1})\right \|^{2}_{2,\infty}\right\}. \]
    Therefore, by using the relation \eqref{R_11} and notion of $\Tilde{\mathcal{H}}_2$, we obtain
    \begin{align*}
       \beta_{k}\leq & \frac{1}{\mathcal{H}_2-\mathcal{H}_3}\bigg\{ \|u_k\|\|\nabla f^{a_{j_{k-1}}}(x_{k-1})-\nabla f^{a_{j_{k-1}}}(x_{k})\|_{2,\infty} \\
        &~~~~~~~~~~~~~~~~~~~\left.-\frac{\mu \mathcal{H}_2}{\mathcal{H}_2-\mathcal{H}_3}\left\|\nabla f^{a_{j_{k-1}}}(x_{k}) -\nabla f^{a_{j_{k-1}}}(x_{k-1})\right \|^{2}_{2,\infty}\right\}
    \end{align*}
     because $\mu>0$ and $\mathcal{H}_2-\mathcal{H}_3 \overset{\eqref{R_12}}{>} 0$. Then, it follows that
     \begin{align*}
         \beta_{k}\leq & \frac{1}{\mathcal{H}_2-\mathcal{H}_3}\bigg\{ \|u_k\|\|\nabla f^{a_{j_{k-1}}}(x_{k-1})-\nabla f^{a_{j_{k-1}}}(x_{k})\|_{2,\infty} \\
        &~~~~~~~~~~~~~~~~~~~\left.+\frac{\mu \lvert\mathcal{H}_2\rvert }{\mathcal{H}_2-\mathcal{H}_3}\left\|\nabla f^{a_{j_{k-1}}}(x_{k}) -\nabla f^{a_{j_{k-1}}}(x_{k-1})\right \|^{2}_{2,\infty}\right\}.
     \end{align*}
    From the relation \eqref{M_2}, we get
    \begin{align*}
        \beta_{k} &\leq \frac{1}{\mathcal{H}_2-\mathcal{H}_3}\bigg\{ \|\nabla f^{a_{j_{k-1}}}(x_{k-1})-\nabla f^{a_{j_{k-1}}}(x_{k})\|\|u_k\|\nonumber\\
        &~~~~~~~~~~~~~~~~~~~\left.+\frac{\mu \lvert\mathcal{H}_2\rvert }{\mathcal{H}_2-\mathcal{H}_3}\left\|\nabla f^{a_{j_{k-1}}}(x_{k}) -\nabla f^{a_{j_{k-1}}}(x_{k-1})\right \|^{2}\right\}.
    \end{align*}
    Therefore, by using the condition \eqref{U_B_u}, we have
    \begin{align*}
        &\beta_k \leq \frac{1}{\mathcal{H}_2-\mathcal{H}_3}\left\{ L u_{b}\|x_{k}-x_{k-1}\|\|u_k\|+\frac{\mu \lvert\mathcal{H}_2\rvert }{\mathcal{H}_2-\mathcal{H}_3}L^{2}\left\|x_{k}-x_{k-1}\right \|^{2}\right\}\\
        \implies & \beta_k\leq \frac{1}{\mathcal{H}_2-\mathcal{H}_3}\left\{ L u_{b}\|x_{k}-x_{k-1}\|+\frac{\mu \lvert\mathcal{H}_2\rvert }{\mathcal{H}_2-\mathcal{H}_3}L^{2}\left\|x_{k}-x_{k-1}\right \|\left(\|x_1\|+\|x_2\|\right)\right\}.
    \end{align*}
     Thus, by using Assumptions \ref{Ass. 3}, we get
     \begin{equation}\label{R_21}
         \beta_{k}\leq \frac{\|x_{k}-x_{k-1}\|}{\mathcal{H}_2-\mathcal{H}_3}\left\{ L u_{b}+\frac{\mu \lvert\mathcal{H}_2\rvert }{\mathcal{H}_2-\mathcal{H}_3}2\gamma L^{2}\right\}.
     \end{equation}

    \noindent
    From standard Wolfe condition \eqref{stan, wol.}, we get 
    \begin{align}
        \mathcal{H}_2-\mathcal{H}_3&\geq (\sigma-1) \mathcal{H}_3\nonumber\\
        &\geq -(1-\sigma)(1-\frac{1}{2\mu}) \underset{j \in [\omega_{k}]}{\max} \phi \left( \nabla f^{a_j}(x_{k-1})^{\top} u_{k-1}\right)\nonumber\\
       & \geq \frac{1}{2}(1-\sigma)(1-\frac{1}{2\mu})\|u_{k-1}\|^{2}\nonumber\\
       & \geq \frac{1}{2}(1-\sigma)(1-\frac{1}{2\mu}) l^{2}_b.\label{R_22}
    \end{align}
    Note that 
    \begin{align*}
        \mathcal{H}_2\geq \sigma(\mathcal{H}_3-\mathcal{H}_2)+\sigma\mathcal{H}_2
        \implies (1-\sigma)\mathcal{H}_2\geq -\sigma(\mathcal{H}_2-\mathcal{H}_3)
        \implies  \mathcal{H}_2\geq \frac{-\sigma}{1-\sigma}(\mathcal{H}_2-\mathcal{H}_3)
    \end{align*}
    and \begin{equation*}
       \mathcal{H}_2\leq \mathcal{H}_2-\mathcal{H}_3 \implies \frac{\mathcal{H}_2}{\mathcal{H}_2-\mathcal{H}_3}\leq 1.
    \end{equation*}
    Thus, \begin{equation}\label{R_23}
        \frac{\lvert\mathcal{H}_2\rvert}{\mathcal{H}_2-\mathcal{H}_3}\leq \max \left\{1,\frac{\sigma}{1-\sigma}\right\}.
    \end{equation}
    From the relations \eqref{R_21}, \eqref{R_22}, and \eqref{R_23}, we obtain
    \[\beta_{k}\leq \frac{2}{(1-\sigma)^{2}\left(1-\frac{1}{2\mu}\right)l^{2}_b}\left\{Lu_{b}+2\mu L^{2}\gamma\max \left\{1,\frac{\sigma}{1-\sigma}\right\}\right\} \|x_{k}-x_{k-1}\|.\]

\end{proof}

In the next lemma, we discuss the convergence of two series associated with the direction $d_{k}$ under the same contrary condition \eqref{u_l_b}.\\

\begin{lemma}\label{Lemma_dir_bdd}
    Suppose \emph{Assumptions \ref{Ass. 3}, \ref{lip_ass}}, and \emph{\ref{bddbe_ass}} and the relation \eqref{u_l_b} hold. Let $\{x_{k}\}$ be a sequence of nonstationary point of the $\mathcal{SOP}_{\ell}$ \eqref{SPL} generated by \emph{Algorithm \ref{algo}}. Then,
    \begin{equation}\label{dir._bdd}
       \sum_{k=1}^{\infty}\frac{1}{\|d_{k}\|^{2}}< +\infty \text{ and }  \sum_{k=1}^{\infty}\|r_{k}-r_{k-1}\|^{2}< +\infty,
    \end{equation}
    where $r_{k}:=\tfrac{d_{k}}{\|d_{k}\|}$.
\end{lemma}

\begin{proof}
    From Theorem \ref{P_S_D}, $d_{k}$ satisfies the sufficient descent condition \eqref{S_d_c}. Therefore, $d_{k}\neq 0$ and $r_{k}$ is well-defined.\\

    \noindent Note that, from the relations \eqref{u_l_b} and \eqref{Zounten.}, and Theorem \ref{P_S_D}, we have
    \begin{align*}
        \sum_{k=0}^{\infty}\frac{1}{\|d_{k}\|^{2}} &\leq \frac{1}{l_b} \sum_{k=0}^{\infty}\frac{\|u_{k}\|^{4}}{\|d_k\|^{2}}\leq \frac{4}{l^{4}_b}\sum_{k=0}^{\infty}\left\{\frac{\underset{j \in [\omega_{k}]}{\max}  \phi \left( \nabla f^{a_{k,j}}(x_{k})^{\top} u_{k} \right) }{\|d_{k}\|}\right\}^{2} \\
        &\leq \frac{4}{(1-\tfrac{1}{2\mu})l^{4}_b}\sum_{k=0}^{\infty}\left\{\frac{\underset{j \in [\omega_{k}]}{\max}  \phi \left( \nabla f^{a_{k,j}}(x_{k})^{\top} d_{k} \right) }{\|d_{k}\|}\right\}^{2}<+\infty.
    \end{align*} 
    Consequently, the first inequality relation of \eqref{dir._bdd} holds. In order to show the second inequality relation of \eqref{dir._bdd} holds, note that $d_{k}:=u_{k}+\beta_{k}d_{k-1}$. Therefore, we have
    \begin{align}
        & r_{k}:=\frac{d_{k}}{\|d_{k}\|}=\frac{u_{k}}{\|d_{k}\|}+\frac{\beta_{k}d_{k-1}}{\|d_{k}\|}=\frac{u_{k}}{\|d_{k}\|}+q_{k}r_{k-1}, \quad \text{where }q_{k}:=\beta_{k}\frac{\|d_{k-1}\|}{\|d_{k}\|},\nonumber\\
        \text{i.e.,} \quad & \|r_{k}-q_{k}r_{k-1}\|=\frac{\|u_{k}\|}{\|d_{k}\|}\overset{\eqref{U_B_u}}{\leq} \frac{u_{b}}{\|d_{k}\|}.\label{BODL1}
    \end{align}
    Here, $q_{k}>0$ and $\|r_{k}\|=1$ for all $k$. Therefore, using relations \eqref{Ine_lemma} and \eqref{BODL1}, we get
    \[\|r_{k}-r_{k-1}\|\leq 2\|r_{k}-q_{k}r_{k-1}\|\leq \frac{2u_{b}}{\|d_{k}\|}.  \]
\noindent
    Hence, using the last inequality relation, it follows that
    \[\sum_{k=1}^{\infty}\|r_{k}-r_{k-1}\|^{2}\leq 4{u_b}^{2}\sum_{k=1}^{\infty}\frac{1}{\|d_{k}\|^{2}}<+\infty,\]
    which completes the proof.
\end{proof}

Next, we will establish the global convergence of the proposed method.\\

\begin{theorem}\label{T_Global}
    Consider \emph{Algorithm \ref{algo}} and assume that \emph{Assumptions \ref{Ass. 3}, \ref{lip_ass},} and \emph{\ref{bddbe_ass}} holds. Then, the relation \eqref{asymptot} is true.
    
\end{theorem}
\begin{proof}
     The proof is by contradiction of \eqref{asymptot} exactly as in \cite[Theorem 2]{extension2020Hager-Zhang}.
\end{proof}

\section{Numerical experiments}\label{sec_5}
This section presents some numerical results to demonstrate the practical performance and effectiveness of the proposed algorithm. The implementation is carried out using MATLAB R2023b, and all computations are performed on a PC equipped with an $11^{\text{th}}$-generation Intel(R) Core(TM) i5-11320H CPU, 8.0 GB of RAM, and the Windows 11 operating system.\\

We compare the proposed HZ conjugate gradient method \ref{algo} with the PRP and HS conjugate gradient methods discussed in \cite{Ghosh CGM}. The details of the experimental setup and the parameters used for the implementation are outlined below.

\begin{itemize}
    \item We explore three distinct types of the ordering cone \(K\) defined as follows:
\begin{align*}
    & K_{1} := \left\{(y_{1},y_{2}, \ldots,y_{m})^{\top}~\middle|~y_{i}\geq 0,~ i\in [m]\right\}, \\
     & K_{2} := \left\{(y_{1},y_{2})^{\top}~\middle|~-\tfrac{1}{3}y_{1}+y_{2}\geq 0,~ 3y_{1}-y_{2}\geq 0\right\}, \text{ and }  \\
     & K_{3} := \left\{(y_{1},y_{2},y_{3})^{\top}~\middle|~y^{2}_{3}\geq 90\left(y^{2}_{1}+y^{2}_{2}\right),~ y_{3}\geq 0\right\}.
\end{align*}
Here, $K_{1}\subset\mathscr{P}(\mathbb{R}^{m})$ and $K_{2}\subset\mathscr{P}(\mathbb{R}^{2})$ are finitely generated cone, while $K_{3}\subset\mathscr{P}(\mathbb{R}^{3})$ is an infinitely generated cone. The motivation for using different types of cones lies in the fact that \( \mathcal{SOP}_{\ell} \) \eqref{SPL} is highly dependent on the choice of the ordering cone \( K \). Employing the same objective function but varying the cone \( K \) can yield different solutions, highlighting the influence of the cone structure on the optimization problem.\\

\item For the cones \( K_{1} \) and \( K_{2} \), we choose \( e = (1, 1, \ldots, 1)^{\top} \). However, for \( K_{3} \), we use \( e = (0, 0, 1)^{\top} \), as \( (1, 1, 1)^{\top} \notin \mathrm{int}(K_{3}) \).\\

\item In Step 2 of Algorithm \ref{algo}, to determine the minimal set \( M_{k} \) at the point \( x_{k} \), we employ a straightforward approach involving pairwise comparisons of the elements in \( F(x_{k}) \).\\

\item We employ the MATLAB solver \texttt{fmincon}, using the \texttt{quadprog} and interior-point algorithms, to solve the optimization problem described in Step~3. Specifically, the \texttt{quadprog} algorithm is used for the problem associated with the cones \( K_1 \) and \( K_2 \), while the interior-point algorithm is applied to the problem associated with the cone \( K_3 \).\\

\item The stopping condition for each method is set as \( \|u_{k}\| \leq 10^{-4} \).\\

\item To implement the line search for determining the step length, we follow the strategy outlined in \cite{Wolfe2019prudent}, using the following parameters: an initial trial step length \( \alpha_{\mathrm{int}} = 1 \), a maximum allowable step length \( \alpha_{\mathrm{max}} = 100 \), and the constants \( \rho = 10^{-4} \) and \( \sigma = 0.1 \).\\

\item We use $\mu=1$ to find scalar conjugate parameter $\beta^{\text{HZ}}_{k}$ from \eqref{BHZ}. \\

\item Currently, in the existing literature of conjugate gradient methods for set optimization, the study in \cite{kumar2024nonlinear} proposed FR and CD methods, and \cite{Ghosh CGM} derived DY, PRP and HS methods. It is reported in \cite{Ghosh CGM} that PRP and HS methods perform better than FR, CD, and DY methods. Thus, in this paper, we report a performance comparison of the proposed HZ method with PRP and HS methods. To compare the existing PRP and HS methods in \cite{Ghosh CGM} with the proposed method, we consider a set of test problems involving unconstrained multi-objective optimization and set optimization problems, as outlined in Table \ref{Test_Problems}. These multi-objective optimization test problems are utilized to construct set-valued objective functions \( F: \mathbb{R}^n \rightrightarrows \mathbb{R}^m \) for \(\mathcal{SOP}_{\ell}\), defined as follows:  
\[
F(x) := f(x) + G(x),
\]
where \( f: \mathbb{R}^n \to \mathbb{R}^m \) is a vector-valued function, and \( G: \mathbb{R}^n \rightrightarrows \mathbb{R}^m \) is a set-valued map. For simplicity, the resulting set optimization problem is referred to by the same name as the corresponding multi-objective optimization test problem.\\

\item Two set-valued maps are defined to construct the set-valued objective function \( F \) in view of forming $\mathcal{SOP}_{\ell}$ \eqref{SPL} using the multi-objectives function. The first map is a bi-objective set-valued map, while the other is a tri-objective set-valued map.\\
 
\noindent 
\textbf{SVM1}:   
Consider the following set-valued function $G:\mathbb{R}^{n}\rightrightarrows \mathbb{R}^{2}$ defined by 
\begin{equation*}
    G(x):=\{g^{1}(x), g^{2}(x), \ldots, g^{p}(x)\},
\end{equation*}
where $g^{j}(x):=\left(g^{j}_{1}(x),g^{j}_{2}(x)\right)^{\top}$, $j\in [p]$, is with the following expression   
\begin{align*}
    & g^{j}_{1}(x):= \sum_{i=1}^{n}\frac{1}{2^{i-1}}\left\{\cos\left(x_{i}+ \tfrac{2\pi(j-1)}{p}\right)+ \sin\left(x_{i}+ \tfrac{2\pi(j-1)}{p}\right)\right\} \text{ and } \\
    & g^{j}_{2}(x):= \cos\left(\sum_{i=1}^{n}x_{i} +  \tfrac{2\pi(j-1)}{p}\right)^2. 
\end{align*}

\noindent
\textbf{SVM2}:   
Consider the following  set-valued function $G:\mathbb{R}^{n}\rightrightarrows \mathbb{R}^{3}$ defined by 
\begin{equation*}
    G(x):=\{g^{1}(x), g^{2}(x), \ldots, g^{p}(x)\},
\end{equation*}
where $g^{j}(x):=\left(g^{j}_{1}(x),g^{j}_{2}(x),g^{j}_{3}(x)\right)^{\top}$,  $j\in [p]$, is with the following expression 
\begin{align*}
&g^{j}_{1}(x) = 
\left(1 + 0.3\sum_{i=1}^{n}\cos(x_i)
      + 0.3\cos\!\left(\tfrac{6\pi(j-1)}{p}
      + 0.2\sum_{i=1}^{n} x_i\right)
\right)
\cos\!\left(\tfrac{2\pi(j-1)}{p}\right)
\ \\
&g^{j}_{2}(x) = 
\left(1 + 0.3\sum_{i=1}^{n}\cos(x_i)
      + 0.3\cos\!\left(\tfrac{6\pi(j-1)}{p}
      + 0.2\sum_{i=1}^{n} x_i\right)
\right)
\sin\!\left(\tfrac{2\pi(j-1)}{p}\right)
\\\
&g^{j}_{3}(x) = 
0.4\,\sin\!\left(\tfrac{10\pi(j-1)}{p}\right)\sum_{i=1}^{n}\sin(x_i)
\\
&\quad+
0.2\sin\!\left(\tfrac{2\pi(j-1)}{p} + 0.2\sum_{i=1}^{n} x_i\right)\left(1 + 0.3\sum_{i=1}^{n}\cos(x_i)
      + 0.3\cos\!\left(\tfrac{6\pi(j-1)}{p}
      + 0.2\sum_{i=1}^{n} x_i\right)
\right).
\end{align*}


\item For each test case, we generated 100 initial points arbitrarily within the box $lb \le x \le ub$ as indicated in the last two columns of Table \ref{Test_Problems} and executed the algorithm for each selected initial point. In the context of each experiment, we calculated the minimum (\textit{min}), mean, and maximum (\textit{max}) values for the following two metrics:
 \begin{itemize}
\item Time: The total runtime of the algorithm (in seconds) needed to meet the stopping criterion for each initial point.

\item Iteration counts: The number of iterations performed by the algorithm before the stopping criterion is satisfied for each initial point.

 \end{itemize}
 
\item All numerical values reported in Table~\ref{performance_table} are rounded to four decimal places. In Example~\ref{exm_1}, however, five decimal places are reported to distinguish successive iterates $x_k$.   
\end{itemize}

In view of the detailed explanation of our proposed algorithm \ref{algo} with the above experimental setup, we solve a $\mathcal{SOP}_{\ell}$ \eqref{SPL} with infinitely generated cone before comparing the methods.\\
\begin{example}\label{exm_1}
  \emph{Consider a $\mathcal{SOP}_{\ell}$ \eqref{SPL} with the following  set-valued map $F:\mathbb{R}\rightrightarrows \mathbb{R}^3$ given by}  
\end{example}
\[F(x):=\left\{f^{1}(x), f^{2}(x), \ldots, f^{5}(x)\right\}\]
associated with the cone $K_3$, where \[
f^{j}(x)
:= \left(
\begin{aligned}
x + \left(\tfrac{j-3}{2}\right)\sin^2(x)\\
\cos(2x) + \tfrac{1}{1 + e^{2x}} + \left(\tfrac{-j+3}{4}\right)\sin^2(x)\\
x \sin(2x) + \left(\tfrac{-j+3}{2}\right)\sin^2(x)
\end{aligned}
\right), \quad j=1,2,\ldots,5.
\]
\noindent In order to execute Algorithm \ref{algo} for any given initial point, we need an explicit expression of the function value $\phi(y),~y:=(y_1,y_2,y_3)^{\top}\in \mathbb{R}^{3}$. Note that the dual cone $$K^{*}_{3}=\left\{(w_1,w_2,w_3)^{\top} \middle | w^2_1+w^2_2\leq 90\, w_3^2, ~w_3\geq 0\right\}$$ and hence the generator of $K_3^{*}$ is $C=\left\{(w_1,w_2,1)^{\top} \middle | w^2_1+w^2_2\leq 90\right\}$ since $e= (0,0,1)^{\top}$. Therefore, we have
\begin{equation}\label{ex_1_phi}
    \phi(y)=y_3+ \sqrt{90\,(y^2_1+y^2_2)}.
\end{equation}

\noindent We choose an initial point
$x_{0}=-10.71$ (from $[-14,-7]$). Then, 
\begin{align*}
    &F(x_{0})=\big\{(-11.63064,   0.61904, 6.71045)^{\top},(-11.17032,   0.38888, 6.25013)^{\top},(-10.71,   \\& 0.15872, 5.78981)^{\top},(-10.24968, -0.071441, 5.32949)^{\top},(-9.78936,  -0.30160, 4.86917)^{\top}\big\}.
\end{align*}
Therefore, from the equations \eqref{minimal_element}, \eqref{card_fun}, and \eqref{partition_set}, we obtain the following
\begin{align*}
    &M_0=\big\{(-11.63064,   0.61904, 6.71045)^{\top},(-11.17032,   0.38888, 6.25013)^{\top},(-10.71,   \\& 0.15872, 5.78981)^{\top},(-10.24968, -0.071441, 5.32949)^{\top},(-9.78936,  -0.30160, 4.86917)^{\top}\big\}, \\&
   \omega_{0}=5, \text{ and } P_{0}=\{1\}\times\{2\}\times\{3\}\times\{4\}\times\{5\}.
\end{align*}
\noindent Now, using equation \eqref{ex_1_phi} and the MATLAB solver \texttt{fmincon} with the interior-point algorithm, we solve the optimization problem in Step 3 of Algorithm \ref{algo}, obtaining \(a_{0}=\{1,2,\ldots,5\}\) and \(u_{0}=-0.42268\).\\

\noindent Note that $\|u_0\|\geq \epsilon=10^{-4}$. Therefore, $x_0$ is not a required solution, and we need to find another alternative point.\\
  
\noindent From Steps 5 and 6, and the above experimental setup for line search, we obtain $d_0=-0.42268$ and $\alpha_0=1.33256$. Accordingly, we have $x_1:=x_0+\alpha_0 \,d_0=-11.27324$.\\

\noindent Proceeding in the same manner, we obtain the following iterative points $x_k$, as given in Table \ref{exam_1_Table}. Note that $\|u_3\|< \epsilon$. Therefore, $x_3$ is a required solution. \\

\begin{table}[ht]
\centering
\begin{tabular}{ l c c c c c c}
\toprule  
 {$x_{k}$} & {$a_k $} &  {$u_k$} &  {$\|u_k\|$} &  {$d_k$} &  {$\alpha_k$} & {$x_{k+1}:=x_{k}+\alpha_{k}d_k$} \\ \midrule 

$x_0=-10.71$   & $\{1,2,\ldots,5\}$   & $-0.42268$   & $0.42268$   & $-0.42268$        & $1.33256$  &   $-11.27324$    \\ \hline

$x_1=-11.27324$   & $\{1,2,\ldots,5\}$   & $-0.01982$   & $0.01982$   & $-1.30860$        & $0.00024$  &   $-11.27356$    \\ \hline

$x_2= -11.27356$  & $\{1,2,\ldots,5\}$   & $-0.00037$   & $0.00037$   & $-0.00037$            & $0.01698$     &   $-11.27357$     \\ \hline
$x_3= -11.27357$ & $\{1,2,\ldots,5\}$   & $-0.00001$   & $0.00001$  & -- & -- & --   \\
 \bottomrule
\end{tabular}
\caption{Computational details of Algorithm \ref{algo} for Example 1}\label{exam_1_Table} 
\end{table}

For illustrative purposes, we plot the curves \(f^{1}(x), f^{2}(x), \ldots, f^{5}(x)\). 
The red cluster of points represents the initial set value \(F(x_{0})\), while the green cluster corresponds to the terminal set value \(F(x_{3})\). 
The blue clusters denote the intermediate set values \(F(x_{1}) \text{ and } F(x_{2})\), as shown in Fig. \ref{example_1_fig}. 
Note that \(a_{3}=\{1,2,\ldots,5\}\) and \(\{f^{5}(x_{3}) - K_{3}\} \cap \{f^{5}(x)\} = \{f^{5}(x_{3})\}\). 
Hence, as also evident from Fig. \ref{example_1_fig}, \(x_{3} = -11.27357\) appears to be a weakly minimal solution and, therefore, a stationary point.\\

\begin{figure}
    \centering
    \includegraphics[width=0.8\linewidth]{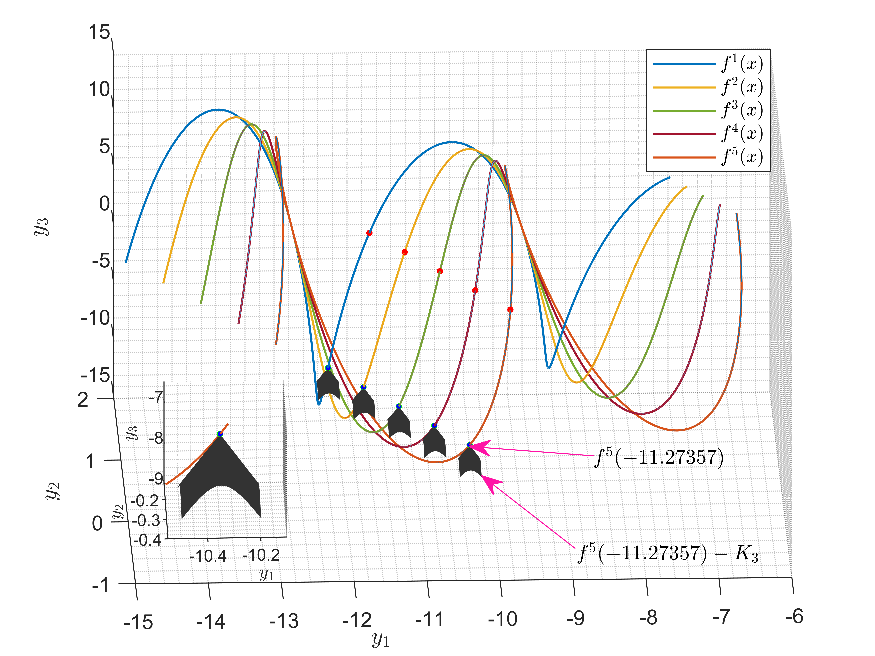}
    \caption{Graphical illustration of Example \ref{exm_1}}
    \label{example_1_fig}
\end{figure}

\begin{table}[ht]
\centering
\caption{Test Problems}\label{Test_Problems} 
\begin{tabular}{ l c c c l l }
\toprule  
 {Problem} & {Reference } &  {$m$} &  {$n$} &  {$lb^{\top}$} &  {$ub^{\top}$} \\ \midrule 
BK1   & \cite{NE_36}   & 2   & 2   & $(-5,-5)$        & $(10,10)$         \\ \hline

DGO2  & \cite{NE_36}   & 2   & 1   & $-9$            & $9$             \\ \hline
DD1  &  \cite{DD-1998}  & 2   & 5   & $(-20,-20,\ldots,-20)$            & $(20,20,\ldots,20)$             \\ \hline
Example \ref{exm_1} & \text{ Newly proposed } & 3   & 1  & $-14$ & $-7$ \\ 
\hline
GAAZ7    & \cite{ghosh2024newton} & 2   & 2   & $(-\pi,-\pi)$        & $(\pi,\pi)$           \\ \hline
GRPY2    & \cite{Ghosh CGM} & 2   & 2   & $(-\pi,-\pi)$ & $(\pi,\pi)$ \\ \hline
Hill  & \cite{NE_34}   & 2   & 2   & $(0,0)$          & $(1,1)$          \\ \hline
IKK1  & \cite{NE_36}   & 3   & 2   & $(-50,-50)$      & $(50,50)$        \\ \hline
JOS1-1 & \cite{NE_38}  & 2   & 10  & $(-2,-2,\ldots,-2)$   & $(2,2,\ldots,2)$   \\ \hline
KGYZ4    & \cite{kumar2024nonlinear} & 2   & 2   & $(-5\pi, -5\pi)$           & $(5\pi, 5\pi)$             \\ \hline
MOP1   & \cite{NE_36} & 2   & 1   & $-10^5$           & $10^5$             \\ \hline
MOP2   & \cite{NE_36} & 2   & 2   & $(-4,-4)$        & $(4,4)$           \\ \hline
MOP7   & \cite{NE_36} & 3   & 2   & $(-400,-400)$    & $(400,400)$       \\ \hline

Toi8    & \cite{Toint1983} & 3   & 3   & $(-1,-1,-1)$ & $(1,1,1)$ \\
\bottomrule
\end{tabular}
\end{table}

In this study, we compare the performance of the proposed HZ method with the HS and PRP methods in \cite{Ghosh CGM}. The results from Table \ref{performance_table} and Fig. \ref{performance_profiles_graph} indicate that the proposed HZ method outperforms HS and PRP. Although, there are some examples where HZ method does not perform well. However, it is important to note that the HZ method provides the descent direction, which ensure its convergence and make it a better choice than the HS and PRP methods.  
To provide deeper insights, we include some visual illustrations in Fig. \ref{G1} that demonstrate how the proposed method operates. 
In Fig. \ref{G1}, we have depicted the movement of iterates from a given initial point to a solution in image space for some problems via the proposed Algorithm \ref{algo}.
Red bunches denote the set value $F(x)$ at the initial point. Blue bunches indicate a set value at intermediate points, and green bunches indicate a set value at a stationary point.

\begin{landscape}
\begin{table}[ht]
\centering
\resizebox{22cm}{!}{
\begin{tabular}{c c c c c c c c c c c c c c c c c c c c c c}
\toprule
Problem & SVM & $K$ & p & \multicolumn{3}{c}{Min time} & \multicolumn{3}{c}{Mean time} & \multicolumn{3}{c}{Max time} & \multicolumn{3}{c}{Min iteration} & \multicolumn{3}{c}{Mean iteration} & \multicolumn{3}{c}{Max iteration} \\
\cmidrule(rl){5-7} \cmidrule(rl){8-10} \cmidrule(rl){11-13} \cmidrule(rl){14-16} \cmidrule(rl){17-19} \cmidrule(rl){20-22}
& & & & PRP & HS & HZ & PRP & HS & HZ & PRP & HS & HZ & PRP & HS & HZ & PRP & HS & HZ & PRP & HS & HZ \\
\midrule
BK1  & SVM1 & $K_{1}$ & 50  & 0.0269 & 0.0261 & 0.0270 & 0.2311 & 0.2107 & 0.1908 & 0.7708 & 1.1262 & 0.5579 & 1 & 1 & 1 & 8.75 & 4.39 & 4.13 & 117 & 11 & 9 \\
BK1  & SVM1 & $K_{2}$ & 50  & 0.0265 & 0.0266 & 0.0325 & 0.0738 & 0.0692 & 0.0863 & 0.2262 & 0.2391 & 0.2340 & 1 & 1 & 1 & 2.46 & 2.46 & 2.50 & 7 & 7 & 7 \\
JOS1 & SVM1 & $K_{1}$ & 50  & 0.2297 & 0.2229 & 0.2355 & 0.4425 & 0.4502 & 0.4988 & 0.7572 & 3.5387 & 2.8638 & 3 & 3 & 3 & 7.30 & 7.42 & 7.16 & 10 & 11 & 11 \\
JOS1 & SVM1 & $K_{2}$ & 50  & 0.2382 & 0.2684 & 0.2657 & 1.0074 & 0.9769 & 1.0515 & 3.8389 & 3.7536 & 3.9401 & 5 & 5 & 5 & 21.78 & 21.80 & 21.36 & 87 & 88 & 85 \\
DD1  & SVM1 & $K_{1}$ & 50  & 0.1770 & 0.1628 & 0.1800 & 0.3992 & 0.3807 & 0.4064 & 0.9733 & 0.8716 & 1.0228 & 5 & 5 & 5 & 9.62 & 9.54 & 9.38 & 17 & 16 & 17 \\
DD1  & SVM1 & $K_{2}$ & 50  & 1.3350 & 1.2555 & 1.3439 & 6.7951 & 6.9030 & 10.7099 & 24.4272 & 22.0062 & 44.6458 & 40 & 39 & 30 & 206.32 & 205.66 & 202.30 & 755 & 717 & 584 \\

Example 1 & -- & $K_{3}$ & 5  & 0.0142 & 0.0135  & 0.0160  & 0.0381 & 0.0402 & 0.0424 & 0.1977 & 0.1520 & 0.1717 & 0 & 0 & 0 & 0.52 & 0.52 & 0.56 & 4 & 4 & 4 \\

DGO2 & SVM1 & $K_{1}$ & 50  & 0.0104 & 0.0104 & 0.0115 & 0.4891 & 0.4295 & 0.4586 & 1.7298 & 1.4644 & 1.8628 & 0 & 0 & 0 & 8.94 & 8.94 & 9.00 & 31 & 31 & 32 \\
GAAZ7 & -- & $K_{1}$ & 100 & 0.0923 & 0.0783 & 0.0746 & 0.2209 & 0.2103 & 0.2114 & 0.4028 & 0.4494 & 0.4799 & 3 & 3 & 3 & 8.01 & 8.14 & 7.93 & 12 & 14 & 13 \\
GAAZ7 & -- & $K_{2}$ & 100 & 0.0861 & 0.1549 & 0.0993 & 0.5959 & 0.7491 & 0.5162 & 2.6516 & 7.8452 & 1.5752 & 3 & 3 & 3 & 8.44 & 13.10 & 9.58 & 26 & 215 & 21 \\
GRPY2 & -- & $K_{1}$ & 100 & 0.0113 & 0.0120 & 0.0137 & 0.2655 & 0.2851 & 0.3382 & 2.2687 & 2.0861 & 3.0987 & 0 & 0 & 0 & 10.02 & 10.24 & 9.34 & 93 & 89 & 84 \\
GRPY2 & -- & $K_{2}$ & 100 & 0.0146 & 0.0147 & 0.0137 & 0.4665 & 0.4725 & 0.3797 & 5.1871 & 4.6953 & 5.0769 & 0 & 0 & 0 & 11.46 & 11.42 & 11.18 & 140 & 137 & 135 \\
Hill  & SVM1 & $K_{1}$ & 50  & 0.0095 & 0.0101 & 0.0092 & 0.1951 & 0.2459 & 0.2003 & 1.0414 & 1.4208 & 1.0702 & 0 & 0 & 0 & 7.38 & 7.35 & 6.93 & 41 & 38 & 34 \\
Hill  & SVM1 & $K_{2}$ & 50  & 0.0164 & 0.0113 & 0.0102 & 0.2666 & 0.1899 & 0.1576 & 2.8456 & 1.9670 & 1.7628 & 0 & 0 & 0 & 4.30 & 5.80 & 5.70 & 41 & 63 & 65 \\
IKK1 & SVM2 & $K_{1}$ & 50  & 0.0109 & 0.0119 & 0.0110 & 0.1268 & 0.1553 & 0.1318 & 0.7470 & 1.1906 & 0.9286 & 0 & 0 & 0 & 2.64 & 2.70 & 2.76 & 17 & 17 & 14 \\
KGYZ4 & -- & $K_{1}$ & 50  & 0.0098 & 0.0091 & 0.0087 & 0.4124 & 0.3411 & 0.3210 & 5.5776 & 4.4038 & 4.0878 & 0 & 0 & 0 & 25.80 & 12.88 & 12.84 & 572 & 205 & 204 \\
MOP1 & SVM1 & $K_{1}$ & 50  & 0.0184 & 0.0208 & 0.0186 & 0.1014 & 0.1517 & 0.0659 & 0.4198 & 0.8113 & 0.1714 & 1 & 1 & 1 & 2.66 & 2.67 & 2.44 & 7 & 7 & 5 \\
MOP1 & SVM1 & $K_{2}$ & 50  & 0.0266 & 0.0278 & 0.0253 & 0.0538 & 0.0614 & 0.0542 & 0.1060 & 0.1180 & 0.1709 & 1 & 1 & 1 & 1.54 & 1.54 & 1.54 & 2 & 2 & 2 \\
MOP2 & SVM1 & $K_{1}$ & 50  & 0.0087 & 0.0087 & 0.0087 & 0.4646 & 0.4644 & 0.1956 & 4.0385 & 3.3501 & 2.8503 & 0 & 0 & 0 & 17.02 & 19.16 & 4.46 & 147 & 147 & 101 \\
MOP7 & SVM2 & $K_{1}$ & 50  & 0.3802 & 0.2778 & 0.1757 & 0.9821 & 0.8243 & 0.6628 & 2.6317 & 2.2910 & 1.4923 & 4 & 3 & 3 & 9.56 & 10.54 & 10.22 & 33 & 26 & 22 \\
Toi8 & SVM2 & $K_{1}$ & 50  & 0.0124 & 0.0133 & 0.0127 & 0.9448 & 0.9807 & 1.3186 & 4.7494 & 4.7555 & 7.8628 & 0 & 0 & 0 & 21.08 & 20.02 & 34.88 & 159 & 130 & 192 \\

\bottomrule
\end{tabular}
}
\caption{Performance metrics of PRP, HS, and HZ on the test problems for a hundred arbitrarily chosen initial points}
\label{performance_table} 
\end{table}

\end{landscape}

\begin{figure}[ht!]
  \centering
    \subfloat[Min computational time]
{\includegraphics[width=0.33\textwidth]{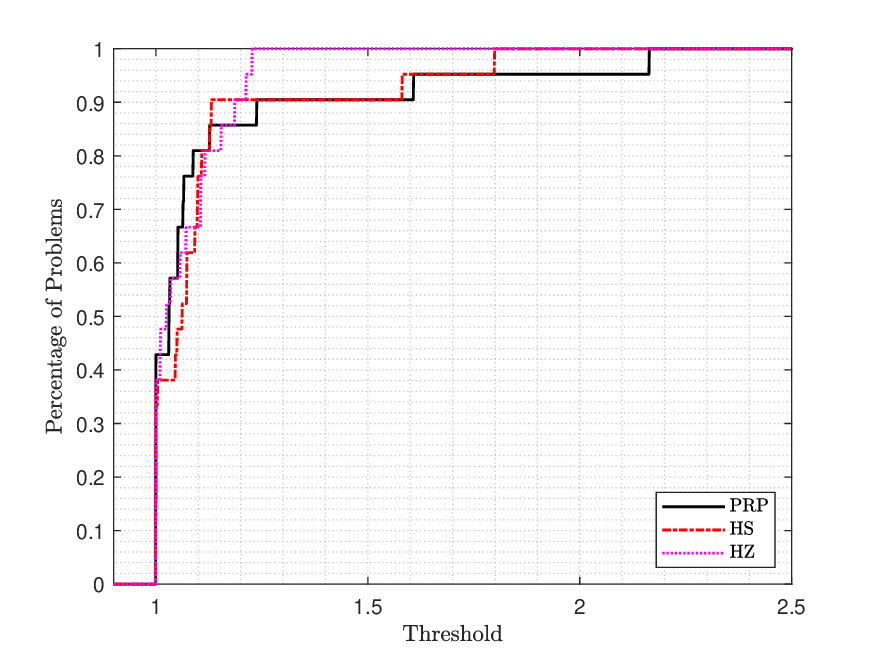}}
\hfill
\subfloat[Mean computational time]
{\includegraphics[width=0.33\textwidth]{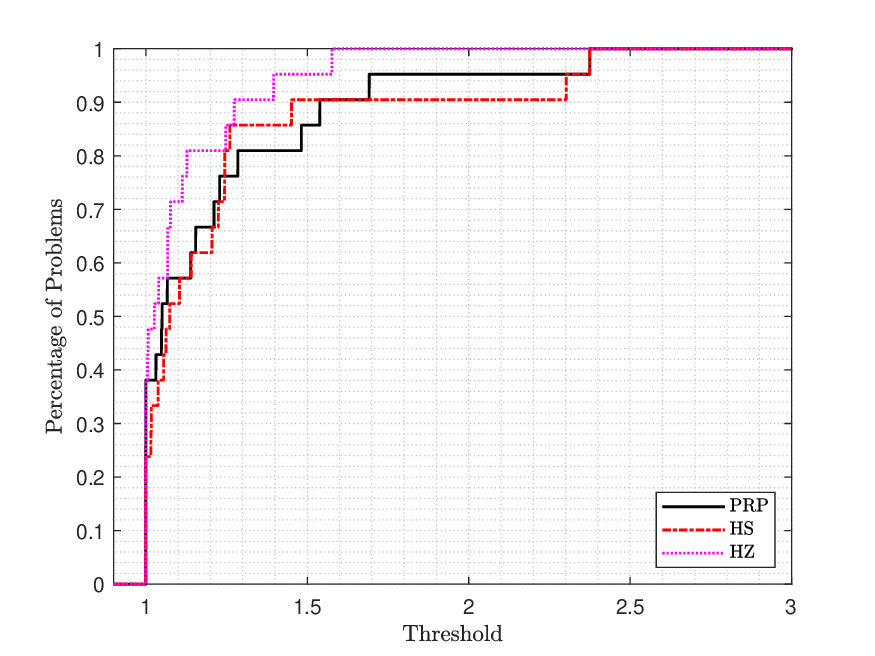}}
\hfill
\subfloat[Max computational time]
{\includegraphics[width=0.33\textwidth]{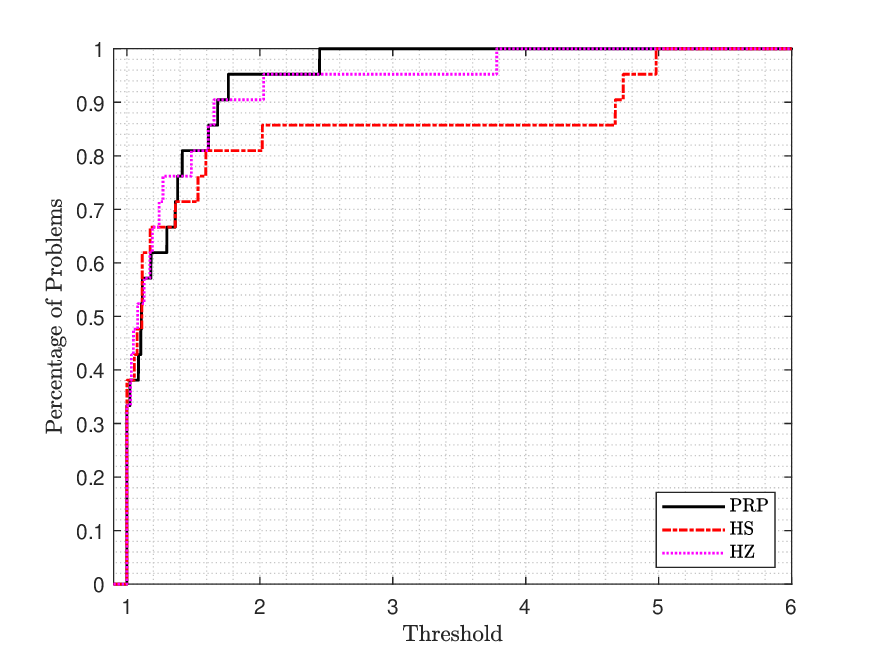}}
\hfill
  \subfloat[Min iteration count]
 {\includegraphics[width=0.33\textwidth]{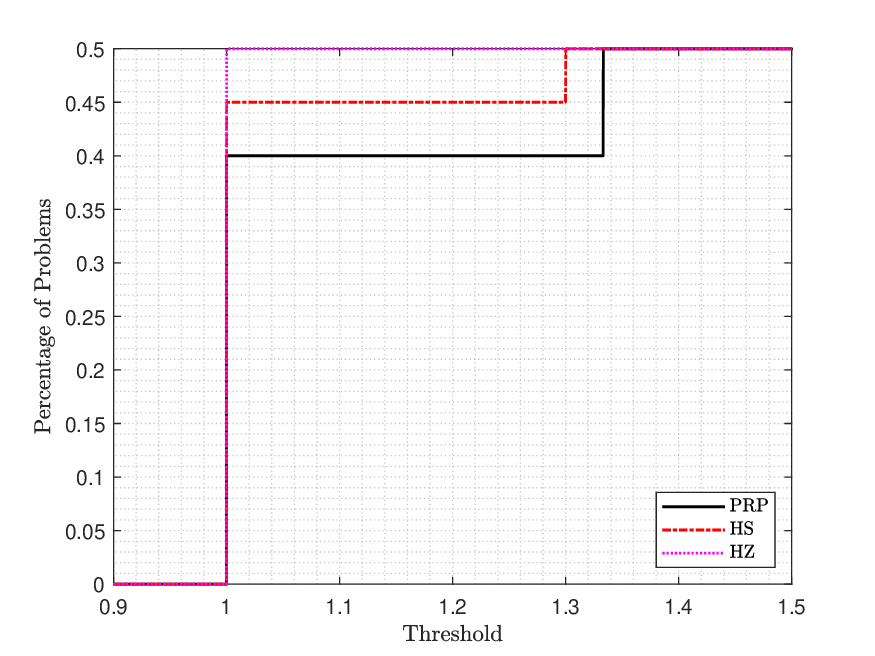}}
\hfill
  \subfloat[Mean iteration count]
{\includegraphics[width=0.33\textwidth]{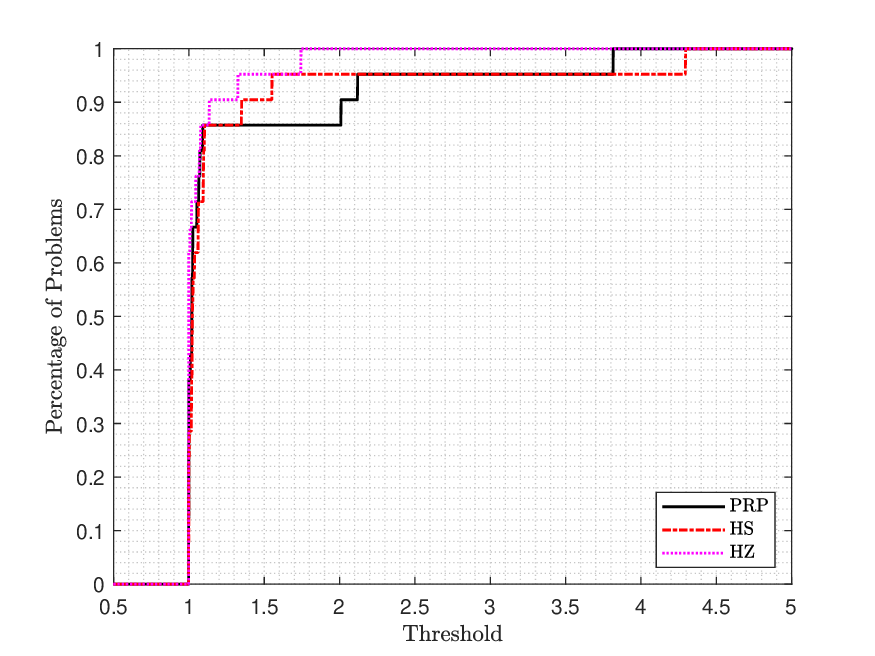}}
\hfill
  \subfloat[Max iteration count]
{\includegraphics[width=0.33\textwidth]{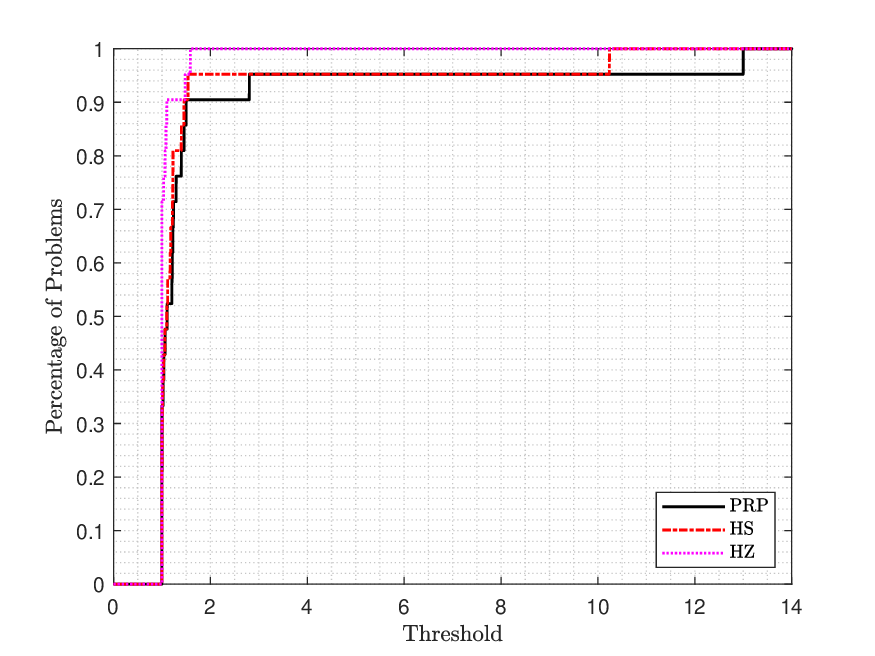}}
\caption{Performance profiles for PRP, HS, and HZ for the test problems given in Table \ref{Test_Problems}}\label{performance_profiles_graph}
\end{figure}

\begin{figure}[ht!]
  \centering
    \subfloat[DD1 with $K_1$ at $x=(-1.9384,0.1313$\\
    $  0.84486,      1.1847,     -1.8853)^{\top}$]{\includegraphics[width=0.5\textwidth]{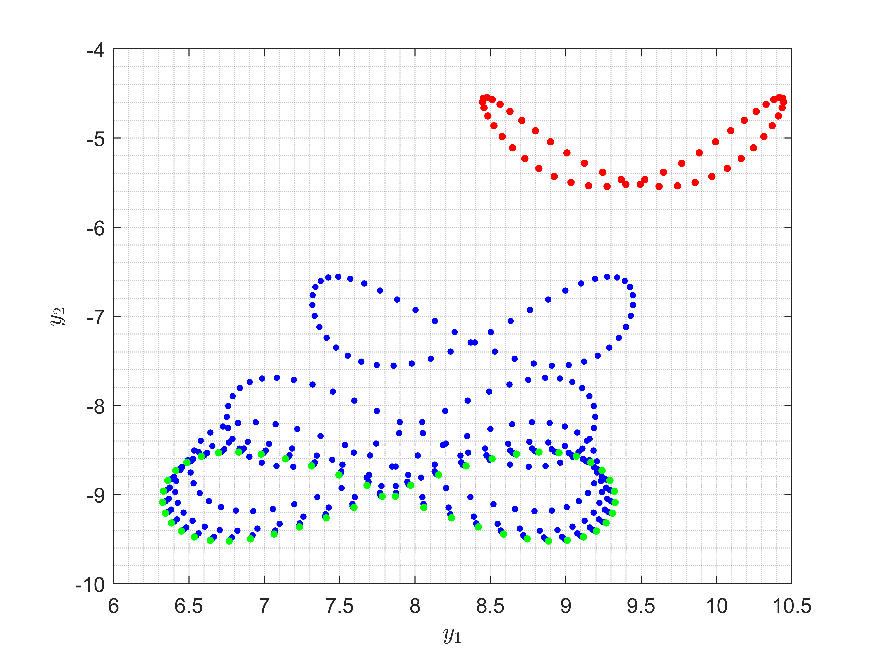}}
    \hfill
  \subfloat[GAAZ7 with $K_1$ at $x=(-0.24677, $\\  $0.19688)^{\top}$]{\includegraphics[width=0.5\textwidth]{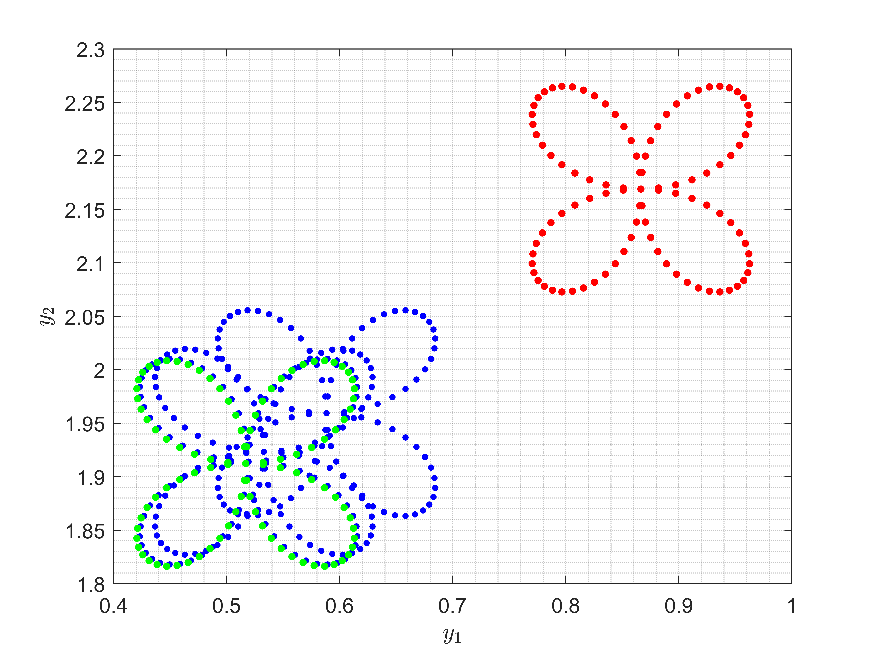}}
  \hfill
  \subfloat[GAAZ7 with $K_2$ at $x=(-0.42945,    $\\$-1.3167 )^{\top}$]{\includegraphics[width=0.5\textwidth]{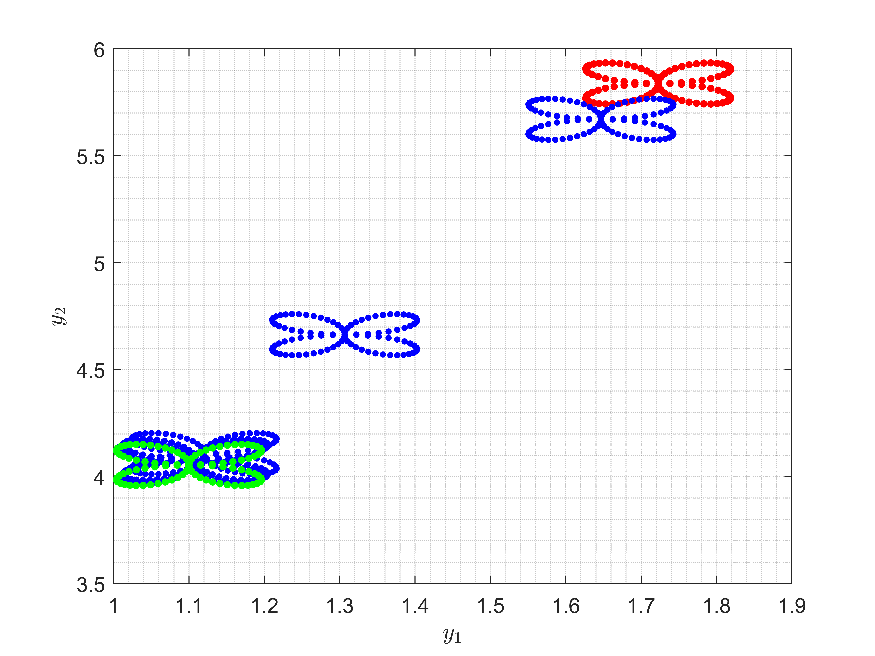}}
    \hfill
  \subfloat[GRPY5 with $K_1$ at $x=(2.4561,$\\ $2.8858)^{\top}$ ]{\includegraphics[width=0.5\textwidth]{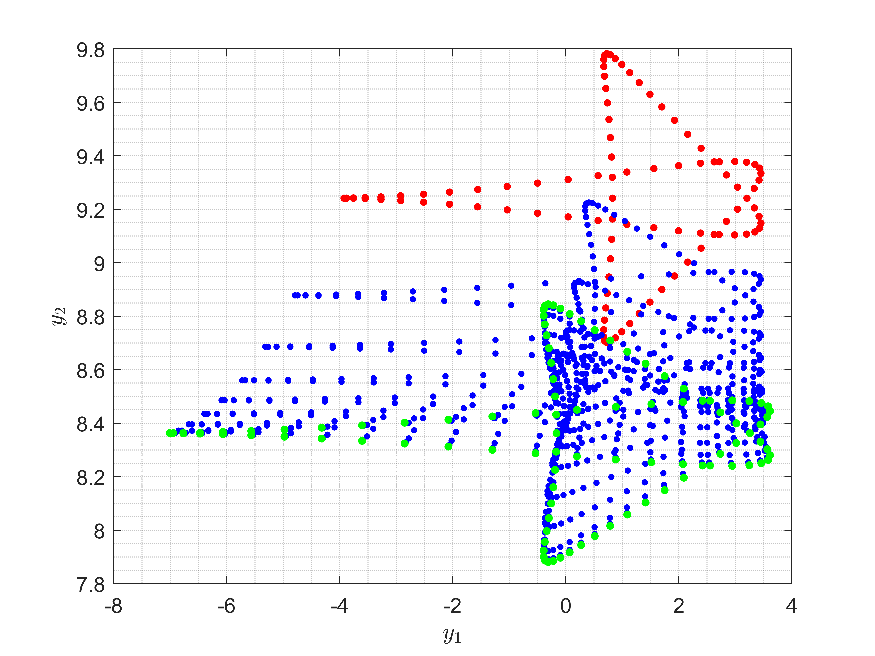}}
  \hfill
  \subfloat[KGYZ4 with $K_1$ at $x=(1.7911,$\\ $1.3603)^{\top}$ ]{\includegraphics[width=0.5\textwidth]{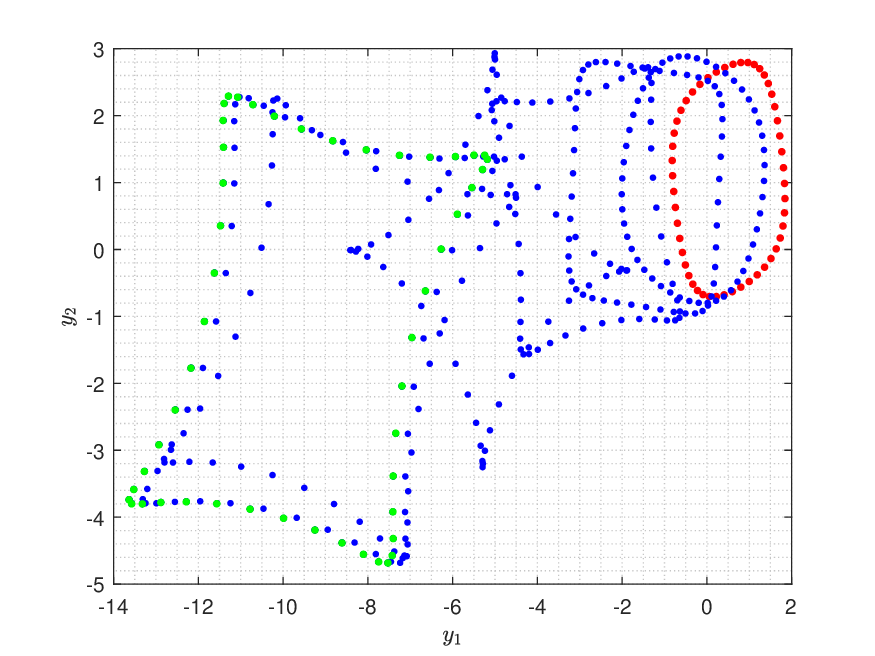}}
    \hfill
  \subfloat[MOP7 with $K1$ at $x=(1.2988 ,    $\\ $-3.2831)^{\top}$]{\includegraphics[width=0.5\textwidth]{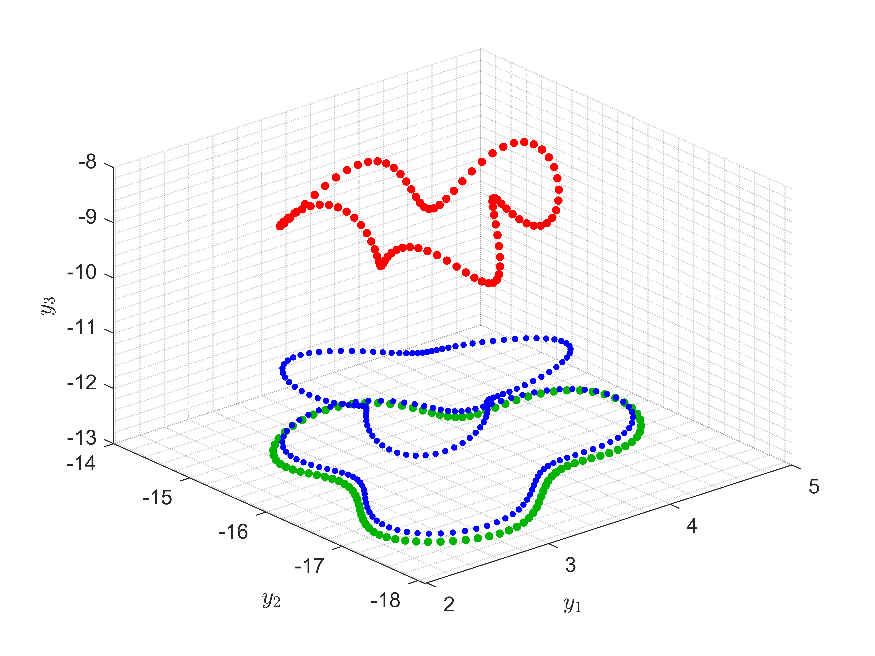}}
  \caption{Sequences $\{F(x_{k})\}$ generated by Algorithm \ref{algo} with HZ rule in the image space for some problem at a given point}{\label{G1}}
\end{figure}

\section{Conclusion}\label{sec_6}
In this work, we have proposed a nonlinear Hager-Zhang conjugate gradient method (Algorithm \ref{algo}) for solving $\mathcal{SOP}_{\ell}$ \eqref{SPL} without requiring regularity conditions on the optimal solution or without assuming that the cone $K$ is finitely generated.
Toward this, we have used the generator $C$ and Drummond-Svaiter scalarization function given in \eqref{gen_polar_cone} and \eqref{Dr_sc_fun}, respectively. Subsequently, the Wolfe line search conditions \eqref{Arm with stan. wol.} and \eqref{Arm with strong. wol.} have also been discussed.
Then, we have established the existence of a step length $\alpha_{k}$ in a $K$-descent direction that holds the strong Wolfe condition \eqref{Arm with strong. wol.} (Theorem \ref{Exist_ST_L}), which ensures that the standard Wolfe line search condition \eqref{Arm with stan. wol.} is also satisfied for the same descent direction $d_{k}$ and step length $\alpha_{k}$.
Thereafter, a scalar conjugate gradient parameter $\beta^{\mathrm{HZ}}_k$ for Hager-Zhang has been introduced in \eqref{BHZ} to define the direction $d_k$ in Algorithm \ref{algo}. Then, the well-definedness of the proposed method is discussed, assuming that the direction $d_{k}$ is $K$-descent. This well-definedness is based on the existence of a point of minima $(a_{k},u_{k})$ of the function $\mathcal{V}_{x}$ given in \eqref{Vxd} and a step length $\alpha_{k}$ which holds the standard Wolfe line search condition \eqref{Arm with stan. wol.} at each iteration $k$. 
The global convergence of the proposed method has been reported (Theorem \ref{T_Global}). In this sequel, we have shown that the direction $d_{k}$, generated by Algorithm \ref{algo}, is $K$-descent (Theorem \ref{P_S_D}). Moreover, we have proved a Zoutendijk-like condition (Theorem \ref{the. zoun.}). Additionally, it is shown that $\{\|u_{k}\|\}$ is bounded (Lemma \ref{Lemma_U_B}). To ensure convergence, we use the contradiction approach in which, firstly, we assume $\|u_{k}\|\geq l_{b}$ for all $k\geq 0$, where $l_b$ is a constant. By using this assumption, we have provided an estimation of the parameter $\beta_{k}$ (Lemma \ref{Lemma_est_Beta}) and proved that the sequence $\{\|d_{k}\|\}$ is unbounded (Lemma \ref{Lemma_dir_bdd}). Then, it has been reported that $\{\|d_{k}\|\}$ is bounded (Theorem \ref{T_Global}). Finally, the practical effectiveness of the proposed method has been validated by comparing its performance with the PRP and HS methods, as detailed in Table \ref{performance_table} and in Fig. \ref{performance_profiles_graph}. These comparisons highlight the advantages of the proposed technique in solving $\mathcal{SOP}_{\ell}$ \eqref{SPL}.\\

In future work, several directions may be explored. One promising direction is applying the proposed methods to solve uncertain optimization problems with finite uncertainty in the direction of the study \cite{Ghosh2024newton}. Additionally, extending our results to other set relations could provide valuable insights into the generality of the approach. Moreover, various other variants of the nonlinear conjugate gradient method could be adapted for $\mathcal{SOP}_{\ell}$, further enhancing the applicability and robustness of the proposed technique.

\subsection*{Acknowledgement}
Debdas Ghosh acknowledges the financial support of the Core Research Grant (CRG/2022/001347) by the Science and Engineering Research Board, India. Ravi Raushan thankfully acknowledges financial support from CSIR, India, through a research fellowship (File No. 09/1217(13822)/2022-EMR-I) to carry out this research work. Zai-Yun Peng was supported by the National Natural Science Foundation of China (12271067).

\subsection*{Data availability}
There is no data associated with this paper.

\end{document}